\documentclass[hidelinks,onefignum,onetabnum]{siamart251216}

\usepackage{amsmath,amssymb,mathtools,amsfonts}
\usepackage{enumitem}
\usepackage{array}

\newsiamremark{assumption}{Assumption}
\newsiamremark{remark}{Remark}
\newsiamremark{example}{Example}
\crefname{assumption}{Assumption}{Assumptions}
\Crefname{assumption}{Assumption}{Assumptions}
\crefname{definition}{Definition}{Definitions}
\Crefname{definition}{Definition}{Definitions}
\crefname{remark}{Remark}{Remarks}
\Crefname{remark}{Remark}{Remarks}
\crefname{example}{Example}{Examples}
\Crefname{example}{Example}{Examples}

\newcommand{\R}{\mathbb R}

\headers{Asynchronous Replanning in LQMFG}{Y. JIN, W. YAO, AND X. ZHANG}
\title{Asynchronous Replanning in Two-Population Linear--Quadratic Mean Field Games: Information Requirements and Stability
\thanks{This work was
supported by the National Natural Science Foundation of China (NSFC Grant
No.~12441101 and 12601852). Corresponding authors: Wang Yao and Xiao Zhang.}}
\author{Yuxin Jin\thanks{Y.~Jin is with ShenYuan Honors College and School of
Mathematical Sciences, Beihang University, Beijing 100191, China; Key
Laboratory of Mathematics, Informatics and Behavioral Semantics, Ministry of
Education, Beihang University, Beijing 100191, China
(\email{yxjin@buaa.edu.cn}).}
\and Wang Yao\thanks{W.~Yao is with School of Artificial Intelligence and
LMIB, Beihang University, Beijing 100191, China; Hangzhou International
Innovation Institute of Beihang University, Hangzhou 311115, China
(\email{yaowang@buaa.edu.cn}).}
\and Xiao Zhang\thanks{X.~Zhang is with the School of Mathematical Sciences,
Beihang University, Beijing 100191, China; Key Laboratory of Mathematics,
Informatics and Behavioral Semantics, Ministry of Education, Beihang
University, Beijing 100191, China (\email{xiao.zh@buaa.edu.cn}).}}

\ifpdf
\hypersetup{
  pdftitle={Asynchronous Replanning in Two-Population Linear-Quadratic Mean Field Games: Information Requirements and Stability},
  pdfauthor={Yuxin Jin, Wang Yao, Xiao Zhang}
}
\fi

\begin{document}
\maketitle

\begin{abstract}
We study asynchronous continuation replanning in a two-population
linear--quadratic mean field game with deterministic open-loop controls
and heterogeneous, possibly erroneous initial information. At exogenous
opportunities, one population recomputes its continuation response from
the physical aggregate state reached, with the opponent's active plan
frozen. The state--plan target needed to initialize replanning is
recoverable from local aggregate observations if and only if a
kernel-inclusion condition holds; full recovery of the hidden initial
input is unnecessary. For revisions accumulating before the terminal
time, spectral stability of the limiting two-response cycle yields
bounded physical execution and a finite state left limit. The
zero-extended continuation plans then converge strongly to the mean field
equilibrium restarted from that state. A scalar counterexample shows
that well-posed continuation problems and a convergent physical state
can coexist with divergent remaining plans. We also construct an exact
causal implementation from local aggregate observations and establish
finite-population robustness on every fixed finite opportunity prefix.
Revision records match with probability tending to one, and the
mean-square implementation error truncated to the matching event is
\(O(N_*^{-1})\), where \(N_*\) is the smaller population
size.
\end{abstract}

\begin{keywords}
linear--quadratic mean field games, functional observability,
continuation responses, asynchronous replanning, spectral stability,
Zeno accumulation, finite-population robustness
\end{keywords}
\begin{MSCcodes}
91A16, 49N10, 93B07, 93D20
\end{MSCcodes}

\section{Introduction}

We consider a two-population linear--quadratic mean field game with
deterministic open-loop controls and heterogeneous initial information.
Each population knows its own initial mean but may use an erroneous
estimate of the opponent's initial mean when computing its initial plan.
After execution has started, the populations receive exogenous
opportunities to revise their remaining plans. At an opportunity time
\(\tau\), the acting population recomputes its continuation response on
\([\tau,T]\) from the physical aggregate state reached, treating the
opponent's currently active continuation plan as frozen. The new plan is
adopted only when it differs from the population's current active plan.

The initial information structure and the revision mechanism lead to two
questions. First, a population does not know the hidden initial belief
that generated the opponent's plan and observes only its own aggregate
component, whereas a continuation response depends on the current
aggregate state and the opponent's remaining plan. The relevant question
is therefore whether the specific state--plan target required by the
continuation response is determined by the local aggregate observation,
rather than whether the hidden initial input can be reconstructed
completely. Second, revisions may occur repeatedly
at different times. The corresponding continuation responses are defined
on the spaces \(L^2([\tau,T])\), whose domains change with the revision
boundary. In particular, if revision times accumulate before \(T\),
regularity of the controls actually executed before the accumulation
time does not by itself imply convergence of the remaining continuation
plans.

The LQ structure makes both questions accessible through the same
continuation-response map. Under the local well-posedness condition
introduced below, the response of population \(m\) has the bounded affine
form
\[
    \mathfrak R_m(\tau,x,v)
    =
    K_m^\tau x+\mathcal J_m^\tau v+h_m^\tau,
\]
where \(x\) is the physical aggregate state at the continuation boundary,
\(v\) is the opponent's active continuation plan, and \(h_m^\tau\) is
the affine offset. The dependence on the boundary state and opponent
plan identifies the information required to compute a response, while
the operators \(\mathcal J_m^\tau\) determine how successive population
responses interact.

\paragraph{Response-relevant information}
The initialization at \(t_0\) recovers the boundary state \(X_{t_0}\)
and the opponent's remaining plan \(u_n^0|_{[t_0,T]}\), with \(n=3-m\).
It need not recover every coordinate of the hidden plan-generating belief.
Let \(\mathcal O_m\) be the residual observation operator and let
\(T_m^{\rm prot}\) map the hidden input to the unknown part of this
state--plan pair. By \cref{thm:protocol-sufficient-observation}, a
bounded initializer recovers the pair from local observations exactly when
\[
    \ker\mathcal O_m
    \subseteq
    \ker T_m^{\rm prot}.
\]
Thus indistinguishable hidden inputs must produce the same target.
Full observability is sufficient but not necessary: in
\cref{prop:singular-gramian-response-example}, both targets are recoverable
despite singular full observability Gramians. The target initializes the
entire protocol; it need not be minimal for one isolated response.

\paragraph{Repeated responses and moving-boundary stability}
A response persists until the opponent changes its plan, so genuine
revisions alternate even when opportunities do not
(\cref{lem:flow-property}). At a fixed continuation boundary, the
mutual-response fixed point is exactly the MFG equilibrium restarted
from that boundary state (\cref{cor:global-implies-response-block}).
The actual asynchronous process is more difficult because both the
physical boundary state and the continuation domain change between
revisions. We compare the plans by extending them by zero to a common
Hilbert space. Suppose genuine revisions accumulate at a pre-terminal
time \(\bar t<T\), with opportunities locally finite before \(\bar t\).
Let \(p\) be the first revising population and \(q=3-p\). The limiting
cycle on the continuation space at \(\bar t\) is
\[
    Q_Z
    :=
    \mathcal J_q^{\bar t}\mathcal J_p^{\bar t}.
\]
The main stability result, \cref{thm:zeno-strategy-completion}, shows that
the sufficient condition
\[
    r(Q_Z)<1
\]
first bounds the responses strongly enough to control physical execution
and obtain a finite state left limit \(X_{\bar t-}\). With this state
limit established, the zero-extended plans converge strongly to the MFG
equilibrium restarted from \(X_{\bar t-}\). The limit is determined by
the state actually reached, which need not lie on the original
complete-information equilibrium trajectory.

Physical execution uses only plan segments before \(\bar t\), whereas
plan convergence concerns the entire unexecuted tails. The scalar model
in \cref{prop:zeno-strategy-counterexample} satisfies both continuation
well-posedness conditions, and every scheduled update is genuine.
Nevertheless, square-integrable execution and a finite state left limit
coexist with divergent continuation-plan norms. Thus well-posed
continuation problems and regular physical execution do not by
themselves ensure that repeated asynchronous replanning converges to
the restarted MFG equilibrium. The counterexample separates
physical-state regularity from stability of the full continuation plans.

\paragraph{Implementation and finite-population robustness}
The stability analysis uses an ideal process with the state and opponent
plan available at each response. Starting from the recovered state--plan
pair, each population maintains these inputs by causal state propagation
and opponent-plan regeneration after announced revisions. Using only
local aggregate observations, its own plan, and the public revision
record, this procedure reproduces the ideal process exactly on every
finite chronological opportunity prefix
(\cref{thm:mean-field-exact-implementation}).

Empirical aggregate fluctuations perturb the reconstructed states and
plans and can change a pass into a revision, or conversely. Replacing
exact equality in the finite-population pass test by a suitably vanishing
tolerance addresses this discrete sensitivity.
\Cref{thm:finite-prefix-robustness} gives revision-record agreement with
probability tending to one on every fixed finite chronological opportunity
prefix, and an \(O(N_*^{-1})\) bound on the mean-square implementation
error weighted by the matching-event indicator. Here \(N_*\) is the
smaller population size. The estimate is not uniform through Zeno
accumulation and makes no \(\varepsilon\)-Nash claim for the revised
finite-player process.

\paragraph{Related work}
Within the dynamic-game and several-population MFG setting
\cite{basar-olsder-1999,lasry-lions-2007,huang-caines-malhame-2007,carmona-delarue-2018,bensoussan-huang-lauriere-2018},
the information question depends on what is hidden and what must be
recovered. Related restrictions arise in partial-observation models
\cite{sen-caines-2019,bensoussan-feng-huang-2021}, incomplete-information
models \cite{bertucci-2022,li-nie-wang-yan-2024}, games with differing
beliefs \cite{casgrain-jaimungal-2020}, robust mean field teams
\cite{feng-huang-jia-2025}, and models without rational expectations
\cite{moll-ryzhik-2026}. These works study different sources of
uncertainty about the state or game environment. Here an exact local
aggregate path must determine a prescribed state--plan target, connecting
the problem to functional observability
\cite{darouach-2000,fernando-trinh-jennings-2010,montanari-duan-aguirre-motter-2022},
bounded operator factorization \cite{douglas-1966}, and classical
observability and identification
\cite{kailath-1980,ljung-1999}. Our earlier paper
\cite{jin-ren-yao-zhang-2026} treats heterogeneous erroneous initial
information, its propagation and Gramian recovery, one-shot revision,
and finite-population approximation of the resulting aggregate.
Repeated replanning adds moving continuation domains, interacting
responses, and propagation of empirical errors into later decisions.

Repeated adjustment has been studied through learning in games
\cite{fudenberg-levine-1998}, asynchronous best replies
\cite{nisan-schapira-zohar-2008}, and asynchronous iteration
\cite{bertsekas-tsitsiklis-1989}. In MFGs, related mechanisms include
fictitious play \cite{cardaliaguet-hadikhanloo-2017},
quasi-stationary and myopic adjustment
\cite{mouzouni-2020,neumann-2024}, and mean field model predictive
control \cite{degond-herty-liu-2017}. The present response starts from
the physical state actually reached, freezes the opponent's current
time-dependent plan, and extends to the original terminal time. These
features distinguish it from updates against history averages or current
population distributions and from short receding-horizon control.

Local implementation raises separate questions about communication and
event records. Communication-constrained MFGs
\cite{aggarwal-zaman-basar-2022}, hybrid MFGs with switching and
stopping \cite{firoozi-pakniyat-caines-2022}, and asynchronous mean field
reinforcement learning \cite{yang-2026} address other forms of
communication, mode switching, or asynchronous learning. Event-triggered
equilibrium seeking \cite{rodrigues-oliveira-krstic-basar-2026} typically
excludes Zeno behavior through a positive dwell time. Our opportunity
sequences are exogenous and may accumulate, which instead leads to the
hybrid-systems question of whether the physical execution and the full
continuation plans admit limits
\cite{liberzon-2003,goebel-sanfelice-teel-2012}.

\paragraph{Organization}
\Cref{sec:model} introduces the model and response maps;
\cref{sec:response-information} characterizes target recovery.
\Cref{sec:recovery} develops the response dynamics, spectral stability,
and counterexample. \Cref{sec:async,sec:finite-population-robustness}
treat local implementation and finite-population robustness, respectively.
The appendices contain the supporting LQ derivations and moving-boundary
estimates.

\section{Model and continuation response maps}\label{sec:model}

Separate global and frozen-opponent boundary systems define the initial
plans and later affine responses, respectively.

\subsection{Baseline two-population LQ MFG}

For populations \(\mathcal P_m\), \(m=1,2\), let
\(H_m=\Pi_m:\mathbb R^{2d}\to\mathbb R^d\) project onto the observed
aggregate block \(z^m\). Given deterministic
\(X\in H^1([0,T];\mathbb R^{2d})\) and
\(\bar u^m\in L^2([0,T];\mathbb R^{d_0})\), a representative state satisfies
\begin{equation}\label{eq:model-state}
dx^m_t=
\bigl(A_mx^m_t+B_mu_m(t)+C_mX_t+F_m\bar u^m_t\bigr)dt
+D_m\,dW^m_t,
\end{equation}
where \(X_t=((z_t^1)^\top,(z_t^2)^\top)^\top\) and
\(C_m=(C_m^1,C_m^2)\).
The admissible class is the deterministic open-loop space
\begin{equation}\label{eq:deterministic-admissible-class}
\mathcal U_m^{\mathrm{det}}:=L^2([0,T];\mathbb R^{d_0}).
\end{equation}
For \(u_m\in\mathcal U_m^{\mathrm{det}}\), the representative cost is
\begin{equation}\label{eq:model-cost}
J_m^{\mathrm{MF}}(u_m;X,\bar u^m)
=\frac12\mathbb E\left[
\int_0^T f_m(t,x^m_t,u_m(t),X_t)\,dt
+G_m(x^m_T,X_T)
\right],
\end{equation}
where
\begin{align*}
f_m(t,x,u,X)
&=\|x-s_m\|_{Q_{Im}}^2+\|u\|_{R_m}^2
+\sum_{n=1}^2\|x-\psi_m^n(X)\|_{Q_m^n}^2,\\
G_m(x,X)
&=\|x-\bar s_m\|_{\bar Q_{Im}}^2
+\sum_{n=1}^2\|x-\bar\psi_m^n(X)\|_{\bar Q_m^n}^2,
\end{align*}
Here \(\|v\|_Q^2:=v^\top Qv\),
\(\psi_m^n(X)=\Gamma_m^nz^n+\eta_m^n\), and
\(\bar\psi_m^n(X)=\bar\Gamma_m^nz^n+\bar\eta_m^n\).
The matrices \(A_m,C_m^n,Q_{Im},Q_m^n,\bar Q_{Im},\bar Q_m^n\) and the target
matrices are \(d\times d\), \(B_m,F_m\in\mathbb R^{d\times d_0}\),
\(R_m\in\mathbb R^{d_0\times d_0}\), and
\(D_m\in\mathbb R^{d\times d_w}\).  All coefficients are deterministic and
constant.

\begin{assumption}[Basic LQ data]\label{ass:lq}
For each \(m\), all state-cost matrices are symmetric nonnegative definite,
\(R_m\) is symmetric positive definite, and the representative initial state
is square integrable and independent of \(W^m\).
\end{assumption}

\begin{definition}[Deterministic mean-field equilibrium]
\label{def:deterministic-mfe}
A deterministic mean-field equilibrium with initial aggregate mean \(X_0\) is
a tuple \((X,u_1,u_2)\) such that each \(u_m\) is the unique optimizer of
\cref{eq:model-state,eq:model-cost} when \(X\) and \(\bar u^m=u_m\) are
treated as exogenous, and
\(X=((\mathbb E x^1)^\top,(\mathbb E x^2)^\top)^\top\) is generated by the
two representative dynamics.  The solution at the true initial aggregate
mean is the \emph{complete-information equilibrium}.
\end{definition}

\begin{assumption}[Global no-conjugate condition]
\label{ass:eq-wellposed}
With the stacked coefficient blocks defined in \cref{app:baseline-lq}, let \(U,V:[0,T]\to\mathbb R^{2d\times2d}\) solve
\begin{equation}\label{eq:global-boundary-lift}
\frac d{dt}\begin{pmatrix}U\\V\end{pmatrix}
=
\begin{pmatrix}
\mathcal A+\mathcal C&-(\mathcal B+\mathcal F)\mathcal R^{-1}\mathcal B^\top\\
-\mathcal Q&-\mathcal A^\top
\end{pmatrix}
\begin{pmatrix}U\\V\end{pmatrix},
\qquad
\begin{pmatrix}U(T)\\V(T)\end{pmatrix}
=\begin{pmatrix}I_{2d}\\\bar{\mathcal Q}\end{pmatrix}.
\end{equation}
Assume \(\det U(t)\ne0\) for every \(t\in[0,T]\).
\end{assumption}

\subsection{Initial plans and local observations}\label{sec:belief}
Under \cref{ass:lq,ass:eq-wellposed}, the equilibrium has the affine
solution map \(\mathcal S:b\mapsto(X[b],Y[b],u_1[b],u_2[b])\) derived in
\cref{app:baseline-lq}. At the true mean \(X_0\), write
\(\mathcal S(X_0)=(X^c,Y^c,u_1^c,u_2^c)\).
Population \(m\) knows its own mean \(z_0^m\), uses a possibly incorrect
opponent mean, and generates its initial plan from
\begin{equation}\label{eq:plan-generating-belief}
b_0^m=X_0+\Pi_n^\top\varepsilon_{m\to n},\qquad
u_m^0=u_m[b_0^m],\qquad n=3-m.
\end{equation}
The actual state is generated by these two plans:
\begin{equation}\label{eq:misspecified-realized-state}
\dot X=\mathsf AX+\mathsf B_1u_1^{\rm act}+\mathsf B_2u_2^{\rm act},\qquad
X(0)=X_0,\quad
\mathsf A:=\mathcal A+\mathcal C,\quad
\mathsf B_m:=\Pi_m^\top(B_m+F_m).
\end{equation}
Initially \(u_m^{\rm act}=u_m^0\). Execution uses the active plan's
current-time value until replacement. All plans have continuous
representatives from the LQ boundary systems.

Fix \(0<t_0<T\) and assume no revision on \([0,t_0]\). Population \(m\)
knows neither \(b_0^n\) nor \(u_n^0\). The affine generator gives
\begin{equation}\label{eq:opponent-plan-generator}
u_n^0=S_nb_0^n+s_n^{\rm pl},\qquad
S_n:\mathbb R^{2d}\longrightarrow L^2(0,T;\mathbb R^{d_0})
\end{equation}
with bounded \(S_n\) and known offset \(s_n^{\rm pl}\). Write \(\Phi^0(t,s)\) for the
transition matrix of \(\mathsf A\). Variation of constants gives
\begin{equation}\label{eq:opponent-belief-representation}
X_t=X_t^{m,\mathrm{base}}+M_m(t)b_0^n,\qquad 0\le t\le t_0,
\end{equation}
where
\begin{equation}\label{eq:base-trajectory-map}
\begin{aligned}
X_t^{m,\mathrm{base}}
&=\Phi^0(t,0)\Pi_m^\top z_0^m
+\int_0^t\Phi^0(t,s)\bigl(\mathsf B_mu_m^0(s)+\mathsf B_ns_n^{\rm pl}(s)\bigr)\,ds,\\
M_m(t)h
&=\Phi^0(t,0)\Pi_n^\top\Pi_nh
+\int_0^t\Phi^0(t,s)\mathsf B_n(S_nh)(s)\,ds.
\end{aligned}
\end{equation}
With population \(m\)'s own mean and plan fixed, the unknown coefficient
\(b_0^n\in\mathbb R^{2d}\) is the hidden input. It is the opponent's
plan-generating initial belief, not the physical state \(X_0\).
Population \(m\) observes only its local aggregate path \(z^m=H_mX\).
Subtracting the known base path gives the observation operator and its
full hidden-input Gramian:
\begin{equation}\label{eq:integral-output-operator}
\begin{aligned}
(\mathcal O_mh)(t)&=H_mM_m(t)h,\qquad
\mathcal O_m:\mathbb R^{2d}\longrightarrow L^2(0,t_0;\mathbb R^d),\\
\mathcal W_m(t_0)&=\mathcal O_m^*\mathcal O_m.
\end{aligned}
\end{equation}
The residual is \(y_m=z^m-H_mX^{m,\mathrm{base}}=\mathcal O_mb_0^n\).
Its role in recovery depends on the inputs required by the response map.

\subsection{Population-wise continuation response maps}\label{sec:response}
Let the continuation plan space be \(\mathbb L(\tau)=L^2([\tau,T];\mathbb R^{d_0})\), with
\(\operatorname{Res}_{b,a}:\mathbb L(a)\to\mathbb L(b)\) denoting restriction
for \(a\le b\). For a boundary state \(x\in\mathbb R^{2d}\) in global order
\((z^1,z^2)\) and a frozen opponent plan \(v\in\mathbb L(\tau)\), define
\(\mathfrak R_m(\tau,x,v)\) by representative-agent optimality followed by
population-\(m\) consistency on \([\tau,T]\). Future revisions are not anticipated in this continuation problem.
The local boundary system is \cref{eq:frozen-fbs}; its rectangular Riccati
chart is independent of the simultaneous equilibrium chart. The
representative agent's stationarity condition involves \(B_m^\top\).
After population consistency is imposed, the population control enters
the aggregate dynamics through \(B_m+F_m\).
The local chart uses actor-first coordinates \((z^m,z^n)\), with
\(n=3-m\). We write \(\mathfrak P_m\) for the permutation from global
to actor-first order, made explicit in
\cref{eq:global-frozen-response-definition} below.
\begin{assumption}[Local frozen-response no-conjugate condition]
\label{ass:response-wellposed}
For \(m=1,2\), let \(U_m\) be the state block of the local graph lift
\cref{eq:local-frozen-lift}.  Assume
\begin{equation}\label{eq:local-frozen-no-conjugate}
\det U_m(t)\ne0,
\qquad t\in[t_0,T].
\end{equation}
By \cref{prop:frozen-no-conjugate}, this condition is equivalent to unique
solvability of every local frozen-response continuation problem.
\end{assumption}

\begin{proposition}[Affine continuation response]
\label{thm:frozen-response}
Assume \cref{ass:lq} and \cref{ass:response-wellposed}.  The response is uniquely defined
and has the affine representation
\begin{equation}\label{eq:affine-response}
\mathfrak R_m(\tau,x,v)
=K_m^\tau x+\mathcal J_m^\tau v+h_m^\tau,
\end{equation}
where
\[
K_m^\tau\in\mathcal L\bigl(
\R^{2d},L^2([\tau,T];\R^{d_0})\bigr),
\qquad
\mathcal J_m^\tau\in\mathcal L\bigl(
L^2([\tau,T];\R^{d_0}),L^2([\tau,T];\R^{d_0})\bigr).
\]
More precisely, if \(\widetilde K_m^\tau\) denotes the locally ordered gain
in \cref{eq:explicit-response-gains}, then
\begin{equation}\label{eq:global-local-response-gain}
K_m^\tau=\widetilde K_m^\tau\mathfrak P_m.
\end{equation}
Moreover,
\begin{equation}\label{eq:uniform-local-response-gains}
\sup_{m=1,2}\sup_{\tau\in[t_0,T]}
\left(
\|K_m^\tau\|+\|\mathcal J_m^\tau\|
+\|h_m^\tau\|_{L^2([\tau,T])}
\right)<\infty.
\end{equation}
The formulas and the proof are given in
\cref{app:local-frozen-details}.
\end{proposition}

In particular,
\begin{equation}\label{eq:exact-response-increment}
\mathfrak R_m(\tau,x_1,v_1)-\mathfrak R_m(\tau,x_2,v_2)
=K_m^\tau(x_1-x_2)+\mathcal J_m^\tau(v_1-v_2).
\end{equation}
The local and global coordinate conventions are
\begin{equation}\label{eq:global-frozen-response-definition}
\mathfrak R_m(\tau,x,v)=\widetilde{\mathfrak R}_m(\tau,\mathfrak P_mx,v),
\qquad \mathfrak P_1=I,\quad
\mathfrak P_2=\begin{pmatrix}0&I_d\\I_d&0\end{pmatrix}.
\end{equation}
Each \(\mathcal J_m^\tau\) is compact, with
\(\|\mathcal J_m^\tau\|\le C(T-\tau)\), by
\cref{prop:response-operator-compact}. Bounds denoted by \(C\) depend only
on the fixed coefficients and finite horizon unless a prefix dependence is
displayed. Neither no-conjugate condition asserts stability of repeated
responses.

\section{Response-relevant information}\label{sec:response-information}

The response maps in \cref{sec:model} require a boundary state and an
opponent plan, which are not directly observed. We characterize when
the initial residual \(y_m=\mathcal O_mb_0^n\) determines this state--plan
pair, and exhibit recoverability with a singular full Gramian.

\subsection{Response-relevant target and factorization}
Fix population \(m\) and let \(n=3-m\). The hidden-input space, observation
space, and protocol-target space are, respectively,
\[
\mathcal H_m:=\mathbb R^{2d},\qquad
\mathcal Y_m:=L^2(0,t_0;\mathbb R^d),\qquad
\mathsf Z_m^{\rm prot}:=\mathbb R^{2d}\times\mathbb L(t_0).
\]
Use the Euclidean norm on \(\mathcal H_m\), the \(L^2\) norm on
\(\mathcal Y_m\), and the product norm
\[
\|(x,w)\|_{\mathsf Z_m^{\rm prot}}^2
:=|x|^2+\|w\|_{\mathbb L(t_0)}^2.
\]
The actual hidden input is \(b_0^n\in\mathcal H_m\); a generic input
\(h\in\mathcal H_m\) produces the residual observation
\(\mathcal O_mh\in\mathcal Y_m\). Define the protocol target by
\begin{equation}\label{eq:protocol-initial-target}
T_m^{\rm prot}:\mathcal H_m\longrightarrow\mathsf Z_m^{\rm prot},\qquad
T_m^{\rm prot}h:=\bigl(M_m(t_0)h,\operatorname{Res}_{t_0,0}S_nh\bigr).
\end{equation}
The unknown state--plan pair is therefore
\begin{equation}\label{eq:protocol-affine-pair}
\bigl(X_{t_0},u_n^0|_{[t_0,T]}\bigr)
=\bigl(X_{t_0}^{m,\mathrm{base}},s_n^{\rm pl}|_{[t_0,T]}\bigr)+T_m^{\rm prot}b_0^n.
\end{equation}
Two inputs \(h_1,h_2\in\mathcal H_m\) are \emph{observationally
indistinguishable} when \(\mathcal O_mh_1=\mathcal O_mh_2\), equivalently
\(h_1-h_2\in\ker\mathcal O_m\). For this target, they are
\emph{response-relevant equivalent} when
\begin{equation}\label{eq:response-relevant-equivalence}
h_1\sim_{\rm resp}h_2\quad\Longleftrightarrow\quad
T_m^{\rm prot}h_1=T_m^{\rm prot}h_2.
\end{equation}
Equivalently, their difference lies in \(\ker T_m^{\rm prot}\).
A \emph{bounded protocol initializer} is a bounded linear map
\(K:\mathcal Y_m\to\mathsf Z_m^{\rm prot}\) satisfying
\(T_m^{\rm prot}=K\mathcal O_m\); it recovers the target from the
residual, and adding the known offset in \cref{eq:protocol-affine-pair}
recovers the state--plan pair.

Full observability \(\mathcal W_m(t_0)\succ0\) is sufficient for target
recovery; the exact criterion only requires indistinguishable inputs to
have the same target.

\begin{theorem}[Response-relevant initialization]\label{thm:protocol-sufficient-observation}
Fix population \(m\). Assume \cref{ass:lq,ass:eq-wellposed} and the representation in
\cref{eq:opponent-belief-representation}. The local observation determines
the state--plan pair in \cref{eq:protocol-affine-pair} for all
\(h\in\mathcal H_m\), equivalently a bounded protocol initializer exists,
if and only if
\begin{equation}\label{eq:protocol-kernel-condition}
\ker\mathcal O_m\subseteq\ker T_m^{\rm prot}.
\end{equation}
Equivalently, observational indistinguishability implies
response-relevant equivalence. Under this condition the canonical bounded
initializer is
\begin{equation}\label{eq:protocol-canonical-factor}
\mathcal K_m^{\rm prot}=T_m^{\rm prot}\mathcal O_m^\dagger,
\qquad T_m^{\rm prot}=\mathcal K_m^{\rm prot}\mathcal O_m,
\end{equation}
where \(\mathcal O_m^\dagger:\mathcal Y_m\to\mathcal H_m\) is the
Moore--Penrose inverse. Its norm is
the smallest observation-to-target amplification among bounded factors:
\begin{equation}\label{eq:protocol-conditioning}
\kappa_m^{\rm prot}:=\|\mathcal K_m^{\rm prot}\|
=\sup_{\mathcal O_mh\ne0}
\frac{\|T_m^{\rm prot}h\|_{\mathsf Z_m^{\rm prot}}}
{\|\mathcal O_mh\|_{\mathcal Y_m}},
\qquad\sup\varnothing:=0.
\end{equation}
It measures the worst-case amplification of an observation error into
the recovered protocol target. For the true residual
\(y_m=\mathcal O_mb_0^n\) and an available residual \(y_m+e\),
with \(e\in\mathcal Y_m\), the target error satisfies
\[
\|\mathcal K_m^{\rm prot}(y_m+e)-T_m^{\rm prot}b_0^n
\|_{\mathsf Z_m^{\rm prot}}
\le\kappa_m^{\rm prot}\|e\|_{\mathcal Y_m}.
\]
\end{theorem}
\begin{proof}
A factorization \(T_m^{\rm prot}=K\mathcal O_m\) implies the kernel
inclusion. Conversely, that inclusion makes
\(\widetilde K(\mathcal O_mh)=T_m^{\rm prot}h\) well defined on
\(\operatorname{Ran}\mathcal O_m\). This range is finite dimensional,
hence closed, and \(\widetilde K\) is bounded. Its extension by zero on
the orthogonal complement is \(T_m^{\rm prot}\mathcal O_m^\dagger\).
The extension has the norm in \cref{eq:protocol-conditioning}; every
initializer agrees with \(\widetilde K\) on the range and therefore has
at least this norm. If \(\mathcal O_m=0\), the kernel inclusion forces
\(T_m^{\rm prot}=0\), so the canonical factor has norm zero. Finally,
the target perturbation is \(\mathcal K_m^{\rm prot}e\), giving the bound.
\end{proof}

This is
functional observability of a prescribed target
\cite{darouach-2000,fernando-trinh-jennings-2010,montanari-duan-aguirre-motter-2022};
the factorization principle is classical \cite{douglas-1966}. The choice of
target identifies the information needed by continuation replanning.

The target \(T_m^{\rm prot}\) initializes the entire continuation
protocol. It need not be minimal for one isolated response, which may
depend only on a lower-dimensional image of this target.

\subsection{Recoverability without full observability}
The distinction between target recovery and full observability in
\cref{thm:protocol-sufficient-observation} occurs even within the present
LQ model, as the following example shows.

\begin{example}[Recoverable state--plan target with singular full Gramian]
\label{prop:singular-gramian-response-example}
Let \(d=d_0=2\), write \(z^m=(a_m,b_m)\), and take any \(0<t_0<T\).
Choose
\begin{equation}\label{eq:singular-example-coefficients}
A_m=F_m=C_m^m=0,\qquad C_m^n=I_2,\qquad
B_m=\operatorname{diag}(1,0),\qquad R_m=I_2,
\end{equation}
where \(n=3-m\). All state-cost weights vanish except
\(\bar Q_{Im}=\operatorname{diag}(1,0)\), with terminal target
\(\bar s_m=0\).

The active dynamics are \(\dot a_1=a_2+u_{1,a}\),
\(\dot a_2=a_1+u_{2,a}\), and the passive dynamics are
\(\dot b_1=b_2\), \(\dot b_2=b_1\). Passive controls vanish, so passive
hidden coordinates have no effect on either plan.
Put \(C_a=\begin{psmallmatrix}0&1\\1&0\end{psmallmatrix}\).
Active adjoints are constant with terminal value \((a_1(T),a_2(T))\).
For \(H=T-\tau\), the active block of the global graph lift is
\[
U_a(H)=\begin{pmatrix}e^H&1-e^H\\1-e^H&e^H\end{pmatrix}.
\]
Its eigenvalues are \(1\) and \(2e^H-1>0\); the corresponding active
local graph block has determinant \(1+\sinh H>0\). Both passive blocks
are \(e^{-C_aH}\), hence invertible. Thus the global and local
no-conjugate conditions hold.

Set \(\alpha=(e^T-1)/(2e^T-1)>0\) and
\(\beta=e^T/(2e^T-1)\). Inverting \(U_a(T)\) shows that population 2's
constant active-plan perturbation is
\(w_2=-\alpha h_{a_1}-\beta h_{a_2}\). For observer 1, variation of constants gives
the two components of its residual:
\[
(\mathcal O_1h)_a(t)=\sinh(t)h_{a_2}+(\cosh(t)-1)w_2,
\qquad
(\mathcal O_1h)_b(t)=\sinh(t)h_{b_2}.
\]
If the residual vanishes in \(L^2(0,t_0)\), these analytic functions
vanish identically. Their derivatives at zero give \(h_{a_2}=h_{b_2}=0\),
and the second active derivative gives \(w_2=0\), hence \(h_{a_1}=0\).
The remaining coordinate \(h_{b_1}\) changes neither the actual opponent
initial-state contribution nor its plan, and is annihilated by
\(T_1^{\rm prot}\). Exchanging populations gives, in global coordinates,
\begin{equation}\label{eq:singular-gramian-response-example}
\begin{aligned}
\ker\mathcal O_1&=\operatorname{span}\{(0,1,0,0)^\top\}
\subseteq\ker T_1^{\rm prot},\\
\ker\mathcal O_2&=\operatorname{span}\{(0,0,0,1)^\top\}
\subseteq\ker T_2^{\rm prot}.
\end{aligned}
\end{equation}
The derivative calculation identifies the kernels; recovery uses the
bounded initializer of \cref{thm:protocol-sufficient-observation}.
Thus both full Gramians are singular, but both targets satisfy
\cref{eq:protocol-kernel-condition} and admit bounded initializers.
\end{example}

The recovered state--plan pair supplies the initial data for asynchronous
replanning in \cref{sec:recovery,sec:async}.

\section{Asynchronous continuation replanning and stability}\label{sec:recovery}

Recovering the initial state--plan pair does not determine the behavior
of repeated revisions. We define the ideal process, identify its
mutual-response fixed point, and establish spectral stability at a
moving-boundary accumulation. A counterexample separates state and
plan convergence.

Replanning begins at \(t_0\) from
\((X_{t_0},u_1^0|_{[t_0,T]},u_2^0|_{[t_0,T]})\), the seed of
\cref{sec:response-information}. The ideal recursion uses the actual
aggregate state and the opponent's active plan at each response.
\Cref{sec:async} reconstructs these inputs from local information by
initialization, causal state propagation, and opponent-plan regeneration.
The current configuration is
\[
\bigl(X_t,u_1^{\rm act}|_{[t,T]},u_2^{\rm act}|_{[t,T]}\bigr),\qquad t\ge t_0,
\]
with actual physical state \(X_t\) and active unexecuted plans; executed portions
are fixed. Let \(\mathcal G_m\subset(t_0,T)\) be population \(m\)'s finite
or countable opportunity sequence, \(n=3-m\).
\begin{assumption}[Opportunities and event delivery]\label{ass:event-grid}
The sequences \(\mathcal G_m\) are strictly increasing, disjoint, and
exogenous to states and candidate plans. Only the actor receives the
current signal; future opportunities are unrevealed. Genuine revisions
are announced immediately to both populations by time and actor. \(\mathcal G_m \bigcap \mathcal G_n=\emptyset\).
\end{assumption}
At \(\tau\in\mathcal G_m\), write \(-/+\) for pre/post-opportunity
values. With equality in \(\mathbb L(\tau)\), the asynchronous update is
\[
u_m^{\rm cand}:=\mathfrak R_m
\bigl(\tau,X_\tau,u_n^{\rm act,-}|_{[\tau,T]}\bigr),
\]
\[
\begin{aligned}
u_m^{\rm act,+}|_{[\tau,T]}&=
\begin{cases}
u_m^{\rm act,-}|_{[\tau,T]},&u_m^{\rm cand}=u_m^{\rm act,-}|_{[\tau,T]}\quad\text{(pass)},\\
u_m^{\rm cand},&\text{otherwise (genuine revision)},
\end{cases}\\
u_n^{\rm act,+}|_{[\tau,T]}&=u_n^{\rm act,-}|_{[\tau,T]},\qquad X_\tau^+=X_\tau^-.
\end{aligned}
\]
Between opportunities, plans stay fixed and \cref{eq:misspecified-realized-state} gives
\[
\dot X_t=\mathsf AX_t+\mathsf B_1u_1^{\rm act}(t)+\mathsf B_2u_2^{\rm act}(t).
\]
When genuine revisions occur, \((\sigma_j,m_j)\) enter the public record:
\begin{equation}\label{eq:event-only-record}
\mathsf{Rec}(t-)=\{(\sigma_j,m_j):\sigma_j<t\},\qquad
\mathsf{Rec}(t)=\{(\sigma_j,m_j):\sigma_j\le t\}.
\end{equation}
Here $\sigma_j$ is the $j$-th revision time, and $m_j$ is the actor. 
\subsection{Persistence and alternation of genuine revisions}
\label{subsec:ideal-benchmark}
The restriction identity below shows why a population that has responded
passes at every later own opportunity until its opponent changes its plan.

\begin{corollary}[Persistence and alternating revisions]\label{lem:flow-property}
Under \cref{ass:lq,ass:response-wellposed}, for $\tau\in\mathcal G_m$, if
\(u_m=\mathfrak R_m(\tau,X_\tau,u_n)\) and \(X\) is generated by
\((u_m,u_n)\), then
\begin{equation}\label{eq:flow-property}
u_m|_{[s,T]}=\mathfrak R_m(s,X_s,u_n|_{[s,T]}),\qquad \tau\le s\le T.
\end{equation}
For the ideal process under \cref{ass:event-grid}, genuine revisions
therefore alternate.
If both populations have a first opportunity and pass there, no event occurs.
Otherwise any first event occurs at its actor's first opportunity.
\end{corollary}
\begin{proof}
Restrict the unique local state--adjoint boundary solution to \([s,T]\).
It has the same terminal condition and opponent input, so local uniqueness
gives \cref{eq:flow-property}. After an adoption, or a pass at which the
current plan already is the response, a population therefore passes at every
later own opportunity until the opponent changes its plan. This proves
alternation and the assertion about the first event. The response identities
follow from the ideal adoption rule.
\end{proof}
Opportunities need not alternate, but genuine revisions do. If the first
event has actor \(p\), write \(q=3-p\), \(\sigma_{2k-1}=s_k\), and
\(\sigma_{2k}=t_k\). Whenever these events exist,
\begin{equation}\label{eq:seed-roles}
m_{2k-1}=p,\qquad m_{2k}=q,
\end{equation}
and the adoption rule gives the alternating recursion
\begin{equation}\label{eq:alternating-response-recursion}
u_p^k=\mathfrak R_p(s_k,X_{s_k},u_q^{k-1}|_{[s_k,T]}),\qquad
u_q^k=\mathfrak R_q(t_k,X_{t_k},u_p^k|_{[t_k,T]}).
\end{equation}

\subsection{Mutual responses and stability of response cycles}
Fix a common boundary \((\tau,x)\), let \(\{p,q\}=\{1,2\}\), and put
\(a_m=K_m^\tau x+h_m^\tau\). One response cycle has the form
\[
u=a_p+\mathcal J_p^\tau v,\qquad
v'=a_q+\mathcal J_q^\tau a_p+Q_\tau v,\qquad
Q_\tau:=\mathcal J_q^\tau\mathcal J_p^\tau.
\]
Setting $v'=v$ gives the mutual-response fixed-point equations
\begin{equation}\label{eq:mutual-response-equations}
u=a_p+\mathcal J_p^\tau v,\qquad v=a_q+\mathcal J_q^\tau u.
\end{equation}

At a boundary \((\tau,x)\), a \emph{simultaneous continuation MFG
equilibrium} is a pair \((u_1,u_2)\in\mathbb L(\tau)^2\), together with
its aggregate trajectory \(X\), such that \(X_\tau=x\), each
representative control is optimal on \([\tau,T]\) with the pair's
aggregate inputs held fixed, and both population consistency conditions
hold. This is the continuation analogue of \cref{def:deterministic-mfe}.
Under \cref{ass:lq,ass:eq-wellposed}, let \(\mathcal E_\tau(x)\) denote
its unique plan pair, with population components \(\mathcal E_{\tau,m}(x)\).

The equilibrium pair is always written in population order:
\begin{equation}\label{eq:continuation-equilibrium-map}
\mathcal E_\tau(x)=(\mathcal E_{\tau,1}(x),\mathcal E_{\tau,2}(x)).
\end{equation}
It is the bounded affine map given by
\cref{eq:equilibrium-system,eq:baseline-plan-formula} on \([\tau,T]\).

Spectral radii of real operators are taken after complexification.
The stability condition \(r(Q_\tau)<1\) will control repeated cycles.

\begin{proposition}[Identification of the continuation MFG equilibrium]
\label{cor:global-implies-response-block}
Under \cref{ass:lq,ass:eq-wellposed,ass:response-wellposed}, for every
\(\tau\in[t_0,T]\) and \(x\in\mathbb R^{2d}\), the mutual frozen-response
equations \cref{eq:mutual-response-equations}, with
\(a_m=K_m^\tau x+h_m^\tau\), have the unique pair
\((\mathcal E_{\tau,p}(x),\mathcal E_{\tau,q}(x))\) in response order.
\end{proposition}
\begin{proof}
The continuation MFG equilibrium satisfies both frozen-response
systems, so its components solve \cref{eq:mutual-response-equations}.
Conversely, take any mutual-response pair and transform the two local
aggregate paths to global order, \(Z_m=\mathfrak P_m^\top X^{[m]}\).
Both solve
\[
\dot Z_m=\mathsf AZ_m+\mathsf B_1u_1+\mathsf B_2u_2,
\qquad Z_m(\tau)=x.
\]
Forward uniqueness gives \(Z_1=Z_2\). The common path and the adjoints
stacked in population order then solve \cref{eq:equilibrium-system}.
By \cref{prop:global-no-conjugate}, that boundary problem has a unique
solution, with plans \(\mathcal E_\tau(x)\).
\end{proof}

In the actual recursion \cref{eq:alternating-response-recursion},
boundaries, states, and response maps vary with \(k\). The next theorem
controls this varying affine recursion on fixed spaces, assuming that
the state inputs converge. Applying it to the physical process will
therefore require a separate state estimate, supplied by
\cref{thm:zeno-strategy-completion}.

\begin{theorem}[Stability of varying affine response cycles]
\label{thm:moving-affine-response}
Let \(\mathcal V,\mathcal H_p,\mathcal H_q\) be real Banach spaces. Suppose
\(x_{p,k},x_{q,k}\to x_*\) in \(\mathcal V\), and consider
\begin{equation}\label{eq:varying-affine-response}
\begin{aligned}
w_p^k&=K_{p,k}x_{p,k}+h_{p,k}+J_{p,k}w_q^{k-1},\\
w_q^k&=K_{q,k}x_{q,k}+h_{q,k}+J_{q,k}w_p^k.
\end{aligned}
\end{equation}
Here \(K_{m,k}:\mathcal V\to\mathcal H_m\) and
\(J_{m,k}:\mathcal H_n\to\mathcal H_m\) are bounded linear maps, with
\(n\ne m\). Assume
\begin{equation}\label{eq:varying-affine-assumptions}
K_{m,k}\to K_m,J_{m,k}\to J_m\hbox{in operator norm},\qquad h_{m,k}\to h_m\hbox{in }\mathcal H_m.
\end{equation}
If \(r(Q)<1\) for \(Q:=J_qJ_p\), then, for every
\(w_q^0\in\mathcal H_q\), the sequence generated by \eqref{eq:varying-affine-response}
satisfies
\[
\|w_p^k-w_p^*\|_{\mathcal H_p}\to0,
\qquad
\|w_q^k-w_q^*\|_{\mathcal H_q}\to0,
\]
where
\begin{equation}\label{eq:abstract-moving-limit}
w_q^*=(I-Q)^{-1}(a_q+J_qa_p),\qquad
w_p^*=a_p+J_pw_q^*,\qquad a_m:=K_mx_*+h_m.
\end{equation}
\end{theorem}
\begin{proof}
Set \(b_{m,k}=K_{m,k}x_{m,k}+h_{m,k}\to a_m\). Eliminating \(w_p^k\)
gives \(w_q^k=Q_kw_q^{k-1}+c_k\), where
\[
Q_k=J_{q,k}J_{p,k}\to Q,\qquad
c_k=b_{q,k}+J_{q,k}b_{p,k}\to c=a_q+J_qa_p.
\]
For \(r(Q)<\rho<1\), the spectral-radius formula makes
\(\|z\|_*:=\sup_{j\ge0}\rho^{-j}\|Q^jz\|\) an equivalent norm with
\(\|Q\|_*\le\rho\). Operator-norm convergence therefore gives
\(\|Q_k\|_*\le\rho'<1\) for \(k>k_0\). With
\(w_q^*=(I-Q)^{-1}c\) and
\(d_k=(Q_k-Q)w_q^*+c_k-c\to0\), iteration yields
\[
\|w_q^k-w_q^*\|_*\le (\rho')^{k-k_0}\|w_q^{k_0}-w_q^*\|_*
+\sum_{j=k_0+1}^k(\rho')^{k-j}\|d_j\|_*.
\]
For any \(J>k_0\), the finite sum over \(j<J\) tends to zero;
the remaining sum is at most
\(\sup_{j\ge J}\|d_j\|_* /(1-\rho')\), which tends to zero as
\(J\to\infty\). Thus \(w_q^k\to w_q^*\); substitution in the first
recursion gives \(w_p^k\to a_p+J_pw_q^*\).
\end{proof}
\subsection{State limits at pre-terminal accumulation}
\label{subsec:conditional-zeno-strategy}
We next separate physical-state convergence from convergence of full
plans. A pre-terminal Zeno accumulation consists of infinitely many
genuine revisions with
\begin{equation}\label{def:finite-zeno-accumulation}
s_k<t_k<s_{k+1},\qquad s_k,t_k\uparrow\bar t<T.
\end{equation}
The process is defined chronologically through finite opportunity prefixes.
Accordingly, \cref{thm:zeno-strategy-completion} assumes that \(\mathcal G_1\cup\mathcal G_2\) is locally finite on
\([t_0,\bar t)\), so that every interval \([t_0,b]\), \(b<\bar t\), contains
only finitely many opportunities. Here the accumulation is
pre-terminal, in the usual hybrid-system sense
\cite{liberzon-2003,goebel-sanfelice-teel-2012}.

The executed control \(u_m^{\rm act}(t)\) is the current-time value of
the active plan. Its integrability alone determines the following state
limit, independently of full-plan convergence.
\begin{proposition}[Executed controls imply a state left limit]
\label{cor:zeno-completion}
Suppose
\[
(u_1^{\rm act},u_2^{\rm act})
\in L^2([t_0,\bar t);\mathbb R^{d_0})^2 .
\]
Then the physical state \(X_t\) has a finite left limit
\(X_{\bar t-}\) as \(t\uparrow\bar t\).
More generally, the same conclusion holds if the executed controls are
only in \(L^1([t_0,\bar t))^2\).
\end{proposition}

\begin{proof}
On a finite interval, \(L^2\subset L^1\). Variation of constants gives
\[
X_t=\Phi^0(t,t_0)X_{t_0}
+\sum_{m=1}^2\int_{t_0}^t\Phi^0(t,s)\mathsf B_m u_m^{\rm act}(s)\,ds.
\]
The transition matrix is bounded up to \(\bar t\), so \(L^1\) controls
make \(X\) bounded. The state equation then has integrable drift;
its integral over \([s,t]\) tends to zero as \(s,t\uparrow\bar t\).
Thus \(X_t\) is Cauchy and has a finite left limit.
\end{proof}

Uniform \(L^2\) bounds on individual plans do not suffice here:
\(\int_I|u|\le |I|^{1/2}\|u\|_{L^2(I)}\), and square roots of the
execution-interval lengths need not be summable. To control execution,
we use the native pointwise bounds of \cref{lem:response-uniform-w1infty}.

\subsection{Convergence to the restarted continuation MFG equilibrium}
\label{subsec:moving-boundary-convergence}

We compare plans on the common Hilbert space \(\mathcal H=\mathbb L(t_0)\).
For \(\tau\in[t_0,T]\), let \(E_\tau:\mathbb L(\tau)\to\mathcal H\)
extend by zero on \([t_0,\tau)\), and let
\(R_\tau=\operatorname{Res}_{\tau,t_0}\). Define
\begin{equation}\label{eq:extended-response-operators}
\widehat K_m^\tau:=E_\tau K_m^\tau,\qquad
\widehat J_m^\tau:=E_\tau\mathcal J_m^\tau R_\tau,\qquad
\widehat h_m^\tau:=E_\tau h_m^\tau.
\end{equation}
Under the local no-conjugate condition \cref{ass:response-wellposed},
\(\widehat K_m^\tau\) and \(\widehat J_m^\tau\) are continuous in
\(\tau\) in operator norm, while \(\widehat h_m^\tau\) is continuous
in \(\mathcal H\); see
\cref{prop:response-boundary-continuity}.

Suppose genuine revisions alternate at \(s_k<t_k<s_{k+1}\uparrow\bar t<T\),
with first actor \(p\) and \(q=3-p\), as in
\cref{eq:alternating-response-recursion,def:finite-zeno-accumulation}.
By \cref{prop:response-boundary-continuity},
\[
Q_k:=\widehat J_q^{t_k}\widehat J_p^{s_k}\longrightarrow
Q_\infty:=\widehat J_q^{\bar t}\widehat J_p^{\bar t}
=E_{\bar t}Q_ZR_{\bar t},\qquad
Q_Z:=\mathcal J_q^{\bar t}\mathcal J_p^{\bar t}
\]
in operator norm. Since \(R_{\bar t}E_{\bar t}=I\), the decomposition
\(\mathcal H=L^2([t_0,\bar t))\oplus E_{\bar t}\mathbb L(\bar t)\)
identifies \(Q_\infty\) with \(0\oplus Q_Z\), so
\(r(Q_\infty)=r(Q_Z)\). The state inputs still require an estimate;
their convergence is a conclusion of the next theorem.

\begin{theorem}[Spectral stability and Zeno completion]
\label{thm:zeno-strategy-completion}
Under the setup above, assume
\cref{ass:lq,ass:eq-wellposed,ass:response-wellposed,ass:event-grid},
and suppose that
\(\mathcal G_1\cup\mathcal G_2\) is locally finite on
\([t_0,\bar t)\).

If
\begin{equation}\label{eq:zeno-tail-small-gain}
r_Z:=r(Q_Z)<1,
\end{equation}
then the executed controls are bounded on \([t_0,\bar t)\), and
the physical state admits a finite left limit
\[
x_*:=X_{\bar t-}.
\]

Moreover, the zero-extended continuation plans converge strongly in
\(\mathcal H=\mathbb L(t_0)\):
\begin{equation}\label{eq:zeno-full-plan-convergence}
\begin{aligned}
E_{s_k}u_p^k
&\longrightarrow
E_{\bar t}\mathcal E_{\bar t,p}(x_*),\\
E_{t_k}u_q^k
&\longrightarrow
E_{\bar t}\mathcal E_{\bar t,q}(x_*).
\end{aligned}
\end{equation}
Hence the replanning process converges at the accumulation time to the
continuation MFG equilibrium restarted from the physical state actually
reached.
\end{theorem}

\begin{proof}
\emph{Step 1: a uniformly stable tail of response cycles.}
Set
\[
w_p^k:=E_{s_k}u_p^k,
\qquad
w_q^k:=E_{t_k}u_q^k,
\]
with \(w_q^0\) denoting the initial active \(q\)-plan embedded in
\(\mathcal H\).  By the alternating recursion and
\cref{eq:extended-response-operators},
\begin{equation}\label{eq:zeno-closed-extended-recursion}
\begin{aligned}
w_p^k
&=
\widehat K_p^{s_k}X_{s_k}
+\widehat h_p^{s_k}
+\widehat J_p^{s_k}w_q^{k-1},\\
w_q^k
&=
\widehat K_q^{t_k}X_{t_k}
+\widehat h_q^{t_k}
+\widehat J_q^{t_k}w_p^k .
\end{aligned}
\end{equation}

Choose \(r_Z<\rho_0<\rho<1\) and set
\[
\|z\|_*:=\sup_{j\ge0}\rho_0^{-j}\|Q_\infty^jz\|_{\mathcal H}.
\]
The spectral-radius formula gives \(\|z\|_{\mathcal H}\le\|z\|_*\le C_*\|z\|_{\mathcal H}\)
and \(\|Q_\infty\|_*\le\rho_0\). Hence some \(K\ge2\) satisfies
\begin{equation}\label{eq:zeno-late-contraction}
\|Q_k\|_*\le\rho<1,
\qquad k\ge K.
\end{equation}

For \(t<\bar t\), define
\[
M(t):=\sup_{t_0\le r\le t}|X_r|,
\qquad
n_k:=\|w_q^k\|_*.
\]
Eliminating \(w_p^k\) from
\cref{eq:zeno-closed-extended-recursion} gives
\[
\begin{aligned}
w_q^k
={}&
Q_kw_q^{k-1}
+\widehat K_q^{t_k}X_{t_k}
+\widehat h_q^{t_k}
+\widehat J_q^{t_k}
\bigl(
\widehat K_p^{s_k}X_{s_k}
+\widehat h_p^{s_k}
\bigr).
\end{aligned}
\]
By the continuity established above and compactness of \([t_0,T]\),
there exists \(C<\infty\) such that, for \(m=p,q\),
\[
\sup_{\tau\in[t_0,T]}
\left(
\|\widehat K_m^\tau\|
+
\|\widehat J_m^\tau\|
+
\|\widehat h_m^\tau\|_{\mathcal H}
\right)
\le C.
\]
  Since \(s_k<t_k\), we have
\(M(s_k)\le M(t_k)\), and hence, for \(k\ge K\),
\begin{equation}\label{eq:zeno-plan-input-bound}
n_k
\le
\rho n_{k-1}
+
C_F\bigl(1+M(t_k)\bigr)
\end{equation}
for some constant \(C_F<\infty\), independent of \(k\) and of the
state values.

Iterating \cref{eq:zeno-plan-input-bound} and using that \(M\) is
nondecreasing gives
\begin{equation}\label{eq:zeno-plan-past-maximum}
n_k
\le
n_{K-1}
+
\frac{C_F}{1-\rho}
\bigl(1+M(t_k)\bigr),
\qquad k\ge K,
\end{equation}
and every state value appearing at an earlier index is bounded by
\(M(t_k)\).

We next obtain the corresponding estimate for the \(p\)-plans.  Since
\(t_{k-1}<s_k\), applying
\cref{eq:zeno-plan-past-maximum} at index \(k-1\) in the first equation
of \cref{eq:zeno-closed-extended-recursion}, and absorbing the single
index \(k=K\) into the constant, yields constants
\(b_0,b_1<\infty\) such that
\[
\|w_p^k\|_{\mathcal H}
\le
b_0+b_1M(s_k),
\qquad k\ge K.
\]
Using the preceding \(L^2\)-bounds for the opponent plans in
\cref{lem:response-uniform-w1infty}, and noting that restriction does not increase the
\(L^2\)-norm, we obtain constants \(c_0,c_1<\infty\) such that
\begin{equation}\label{eq:zeno-native-response-past-bounds}
\begin{aligned}
\|u_p^k\|_{W^{1,\infty}([s_k,T])}
&\le
c_0+c_1M(s_k),\\
\|u_q^k\|_{W^{1,\infty}([t_k,T])}
&\le
c_0+c_1M(t_k),
\qquad k\ge K.
\end{aligned}
\end{equation}

\emph{Step 2: physical-state bound.}
On \([s_k,t_k)\), the executed pair is \((u_p^k,u_q^{k-1})\), whose
bounds involve only \(M(s_k)\) and \(M(t_{k-1})\). On
\([t_k,s_{k+1})\), it is \((u_p^k,u_q^k)\), involving \(M(s_k)\)
and \(M(t_k)\). All these times precede the execution time.
Initial equilibrium-generated plans are continuous, and adopted responses
satisfy \cref{lem:response-uniform-w1infty}; local finiteness makes the
initial prefix bounded. Absorbing that prefix and \(k\le K\) into
constants gives \(U_0,U_1<\infty\) such that
\begin{equation}\label{eq:zeno-executed-causal-bound}
|u_p^{\rm act}(t)|
+
|u_q^{\rm act}(t)|
\le
U_0+U_1M(t),
\qquad
\text{for a.e. }t\in[t_0,\bar t).
\end{equation}

Using \cref{eq:zeno-executed-causal-bound} in the physical state
equation gives constants \(\gamma_0,\gamma_1\ge0\) such that
\[
|X_t|
\le
|X_{t_0}|
+
\gamma_0(t-t_0)
+
\gamma_1\int_{t_0}^t M(r)\,dr,
\qquad t<\bar t.
\]
Taking the maximum on the left and applying Gronwall's
inequality on each interval \([t_0,b]\), \(b<\bar t\), yields
\begin{equation}\label{eq:zeno-state-uniform-bound}
M(t)
\le
\bigl(
|X_{t_0}|+\gamma_0(T-t_0)
\bigr)
e^{\gamma_1(T-t_0)},
\qquad
t<\bar t.
\end{equation}
The bound is independent of \(b\), so \(M\) and, by
\cref{eq:zeno-executed-causal-bound}, the executed controls are uniformly
bounded. The native \(W^{1,\infty}\) plan norms are also uniformly bounded
by \cref{eq:zeno-native-response-past-bounds}. The bounded physical drift
gives \(|X_t-X_s|\le L_X|t-s|\) for \(t_0\le s<t<\bar t\).
Hence \(X_t\) is Cauchy and admits a finite left limit \(x_*\), with
\(X_{s_k},X_{t_k}\to x_*\).

\emph{Step 3: convergence and identification of the continuation plans.}
The state inputs in \cref{eq:zeno-closed-extended-recursion} now
converge to \(x_*\). By \cref{prop:response-boundary-continuity}, the
extended \(K,J\) operators converge in operator norm and the extended
affine terms converge in \(\mathcal H\). Since \(r(Q_\infty)<1\),
\cref{thm:moving-affine-response} gives strong convergence to the unique
limiting affine fixed point.

By \cref{cor:global-implies-response-block}, the mutual-response fixed
point at the boundary \((\bar t,x_*)\) is precisely the restarted
continuation MFG equilibrium.  Therefore
\[
E_{s_k}u_p^k
\longrightarrow
E_{\bar t}\mathcal E_{\bar t,p}(x_*),
\qquad
E_{t_k}u_q^k
\longrightarrow
E_{\bar t}\mathcal E_{\bar t,q}(x_*),
\]
which is \cref{eq:zeno-full-plan-convergence}.
\end{proof}

The limit is restarted from the physical state actually reached, which
need not lie on the original complete-information equilibrium trajectory.
\Cref{prop:response-operator-compact} gives a coefficient-level sufficient condition for
\cref{eq:zeno-tail-small-gain}. In particular, for fixed model data,
the spectral-stability condition holds when the remaining horizon is
sufficiently short.

\subsection{A counterexample: convergent state and divergent plans}

Continuation well-posedness alone does not control the remaining plans:
the following model has an unstable limiting response cycle.

\begin{proposition}[State convergence does not imply plan convergence]
\label{prop:zeno-strategy-counterexample}
There is a scalar two-population LQ model satisfying the global and local
continuation well-posedness assumptions, with a unique continuation MFG
equilibrium at every boundary, and a pre-terminal Zeno revision sequence
in which every scheduled update is genuine. The executed pair satisfies
\[
(u_p^{\rm act},u_q^{\rm act})
\in L^2([t_0,\bar t);\mathbb R)^2,
\]
so \(X_{\bar t-}\) exists, whereas the continuation plans satisfy
\[
\|u_p^k\|_{L^2([\bar t,T])}
+\|u_q^k\|_{L^2([\bar t,T])}
\longrightarrow\infty.\]
In particular, they do not converge to the continuation MFG equilibrium
restarted from \(X_{\bar t-}\).
\end{proposition}
\begin{proof}
Take \(p=1\), \(q=2\), \(d=d_0=1\), \(T=2\), representative-agent dynamics
\(\dot x_p=u_p,\ \dot x_q=u_q\), and \(R_p=R_q=1\), with no
running state cost. With population means
\(z_m=\mathbb E x_m\) and aggregate state \(X=(z_p,z_q)^\top\),
the representative-agent terminal costs are
\[
G_p(x_p,X)=\bigl(x_p(T)-4z_q(T)\bigr)^2,\qquad
G_q(x_q,X)=\bigl(x_q(T)+4z_p(T)\bigr)^2.
\]
The representative optimizes with the aggregate path fixed; population
consistency is imposed after stationarity. Applying the frozen-response
system \cref{eq:frozen-fbs} gives, for
\(x=(z_p(\tau),z_q(\tau))^\top\), \(H=2-\tau\), and opponent plan \(v\),
\begin{equation}\label{eq:zeno-counterexample-responses}
\begin{aligned}
\mathfrak R_p(\tau,x,v)
&=-\frac{z_p(\tau)-4z_q(\tau)-4\int_\tau^2 v(r)\,dr}{1+H}
  \,\mathbf 1_{[\tau,2]},\\
\mathfrak R_q(\tau,x,v)
&=-\frac{z_q(\tau)+4z_p(\tau)+4\int_\tau^2 v(r)\,dr}{1+H}
  \,\mathbf 1_{[\tau,2]}.
\end{aligned}
\end{equation}
For constant opponent plans, the scalar response gains are
\(\lambda_p(H):=4H/(1+H)\) and
\(\lambda_q(H):=-4H/(1+H)\). Every frozen response is constant on its
continuation interval.

For the state block \(U(\tau)\) of the global boundary lift
\cref{eq:global-boundary-lift}, and the local frozen-response state blocks
\(U_m(\tau)\) of \cref{eq:local-frozen-lift},
\[
\det U(\tau)=(1+H)^2+16H^2>0,\qquad
\det U_p(\tau)=\det U_q(\tau)=1+H>0.
\]
Hence \cref{ass:eq-wellposed,ass:response-wellposed} hold at every
continuation boundary by
\cref{prop:global-no-conjugate,prop:local-frozen-riccati}.
The mutual frozen-response pair is the unique continuation MFG
equilibrium by \cref{cor:global-implies-response-block}.
The explicit state blocks and scalar fixed-point calculation are in
\cref{app:divergence-details}.

At \(\bar t=1\), the rank-one operator
\(Q_Z=\mathcal J_q^1\mathcal J_p^1\) satisfies
\(Q_Z\mathbf1_{[1,2]}=-4\mathbf1_{[1,2]}\), so \(r(Q_Z)=4>1\).

We choose shrinking execution intervals for geometrically growing
continuation amplitudes. Let \(t_0<1-64^{-1}\) and
\[
s_k=1-64^{-k},\qquad
t_k=1-\frac12\,64^{-k},\qquad k\ge1.
\]
Then \(s_k<t_k<s_{k+1}\) and \(s_k,t_k\uparrow1\).
Choose \(b_0\ne0\), \(X_0=(0,-s_1b_0)^\top\), and the initial
plans \(u_p^0=0\), \(u_q^0=b_0\), with no opportunities before
\(s_1\). These plans are generated by beliefs of the form
\cref{eq:plan-generating-belief}, as verified in
\cref{app:divergence-details}. They give \(X_{s_1}=0\), and the first
\(p\)-response is nonzero. Before \(s_k\),
\(u_q^{k-1}\equiv b_{k-1}\); the response at \(s_k\) is
\(u_p^k\equiv a_k\) on \([s_k,2]\), and that at \(t_k\) is
\(u_q^k\equiv b_k\) on \([t_k,2]\).
Put \(H_{p,k}=2-s_k\) and \(H_{q,k}=2-t_k\).
Then \cref{eq:zeno-counterexample-responses} gives
\begin{equation}\label{eq:zeno-counterexample-amplitudes}
\begin{aligned}
a_k
&=\lambda_{p,k}b_{k-1}+d_{p,k},
&\lambda_{p,k}
&=\frac{4H_{p,k}}{1+H_{p,k}},
&d_{p,k}
&=-\frac{z_p(s_k)-4z_q(s_k)}{1+H_{p,k}},
\\
b_k
&=\lambda_{q,k}a_k+d_{q,k},
&\lambda_{q,k}
&=-\frac{4H_{q,k}}{1+H_{q,k}},
&d_{q,k}
&=-\frac{z_q(t_k)+4z_p(t_k)}{1+H_{q,k}}.
\end{aligned}
\end{equation}
The executed pair is \((a_k,b_{k-1})\) on \([s_k,t_k)\) and
\((a_k,b_k)\) on \([t_k,s_{k+1})\); the full plans extend to \(T=2\).

For the physical boundary-state size
\(S_k:=|z_p(s_k)|+|z_q(s_k)|\), the induction in
\cref{app:divergence-details} proves
\begin{equation}\label{eq:zeno-counterexample-cone}
S_k\le0.1|b_{k-1}|,\qquad
3|b_{k-1}|<|b_k|<5.1|b_{k-1}|.
\end{equation}
In particular,
\[
|b_k|>3^k|b_0|\longrightarrow\infty,
\qquad
|a_k|\ge1.8\,3^{k-1}|b_0|\longrightarrow\infty.\]
The estimates in \cref{app:divergence-details} also give
\[
\operatorname{sgn}(b_k)=-\operatorname{sgn}(b_{k-1}),
\qquad
\operatorname{sgn}(a_k)=-\operatorname{sgn}(a_{k-1}) \quad (k\ge2),
\]
so every scheduled update is genuine. Both plans are constant on
the common tail \([1,2]\); hence
\begin{equation}\label{eq:zeno-counterexample-plan-divergence}
\|u_p^k\|_{L^2([1,2])}
=|a_k|\longrightarrow\infty,
\qquad
\|u_q^k\|_{L^2([1,2])}
=|b_k|\longrightarrow\infty.
\end{equation}
It remains to check the executed controls. By
\cref{eq:zeno-counterexample-cone} and the bound for \(a_k\) in
\cref{app:divergence-details},
\(|b_k|\le |b_0|5.1^k\) and
\(|a_k|\le2.3|b_0|5.1^{k-1}\).
Since \(s_{k+1}-s_k=(63/64)64^{-k}\), the controls executed on
\([s_k,s_{k+1})\) satisfy
\[
\int_{s_1}^{1}
 \bigl(|u_p^{\rm act}(t)|^2+|u_q^{\rm act}(t)|^2\bigr)\,dt
\le C\sum_{k\ge1}64^{-k}5.1^{2k}<\infty,
\qquad \frac{5.1^2}{64}<1.
\]
The initial controls on \([t_0,s_1)\) are restrictions of admissible
plans. Hence
\((u_p^{\rm act},u_q^{\rm act})\in L^2([t_0,1);\mathbb R)^2\),
and \cref{cor:zeno-completion} gives the finite physical-state
left limit \(X_{1-}\). The continuation plans nevertheless diverge,
so they cannot converge to the restarted continuation MFG
equilibrium. 
\end{proof}

\section{Local implementation from aggregate observations}\label{sec:async}
The ideal process requires the current state and opponent plan. Local
aggregate observations, the population's own plan, and public revision
events allow each population to recover and maintain these inputs.
The resulting implementation agrees with the ideal process on every
finite chronological opportunity prefix.
\subsection{Available local information and initialization}
\label{subsec:information-structure}
At time \(t-\), population \(m\) has its observed aggregate path
\(z^m_{[0,t]}=H_mX_{[0,t]}\), its own active continuation plan
\(u_m^{\mathrm{act}}|_{[t,T]}\), and the public revision record
\(\mathsf{Rec}(t-)\). We collect these quantities as
\begin{equation}\label{eq:operational-information-state}
\mathfrak I_{t-}^m=
\bigl(z^m_{[0,t]},u_m^{\rm act}|_{[t,T]},\mathsf{Rec}(t-)\bigr).
\end{equation}
At \(t_0\), population \(m\) forms the residual
\(y_m=z^m-H_mX^{m,\mathrm{base}}\) from \cref{eq:base-trajectory-map}.
\begin{assumption}[Initial recoverability]
\label{ass:protocol-observability}
There is no revision on \([0,t_0]\). For each population, the base terms
in \cref{eq:protocol-affine-pair}, the observation map \(\mathcal O_m\)
in \cref{eq:integral-output-operator}, and the protocol target
\(T_m^{\rm prot}\) in \cref{eq:protocol-initial-target} are known, and
\(\ker\mathcal O_m\subseteq\ker T_m^{\rm prot}\).
\end{assumption}
Under this assumption, by \cref{thm:protocol-sufficient-observation}, the initializer
\(\mathcal K_m^{\rm prot}\) in \cref{eq:protocol-canonical-factor}
recovers the state--plan pair in \cref{eq:protocol-affine-pair}:
\begin{equation}\label{eq:protocol-initializer}
\bigl(\widehat X_{t_0}^m,\widehat u_{n|m}^{\rm act}\bigr)
=\bigl(X_{t_0}^{m,\mathrm{base}},s_n^{\rm pl}|_{[t_0,T]}\bigr)
+\mathcal K_m^{\rm prot}y_m
=\bigl(X_{t_0},u_n^0|_{[t_0,T]}\bigr).
\end{equation}
\subsection{Causal propagation and event updates}\label{subsec:causal-propagation}
Write \(\widehat X_t^m\) and \(\widehat u_{n|m}^{\rm act}\) for the
locally maintained state and opponent plan, with \(n=3-m\).

\emph{Between revision events.}
The estimate \(\widehat X_t^m\) of the physical aggregate state
\(X_t\) evolves by the causal integral observer
\begin{equation}\label{eq:causal-integral-observer}
\begin{aligned}
\widehat X_t^m=\widehat X_{t_0}^m+\int_{t_0}^t\bigl(\mathsf A\widehat X_s^m+
\mathsf B_mu_m^{\rm act}(s)+\mathsf B_n\widehat u_{n|m}^{\rm act}(s)\bigr)\,ds+\int_{t_0}^tL_m(s)\bigl(z_s^m-H_m\widehat X_s^m\bigr)\,ds,
\end{aligned}
\end{equation}
Here \(\widehat u_{n|m}^{\rm act}\) is population \(m\)'s stored opponent
active continuation plan, \(z^m=H_mX\) is its observed local aggregate,
and \(L_m:[t_0,T]\to\mathbb R^{2d\times d}\) is a known deterministic
bounded measurable observer gain; \(L_m=0\) is allowed. Between revisions,
the stored plan remains fixed while its unexecuted restriction shortens
with time. Exact initialization in \cref{eq:protocol-initializer} makes
\(\widehat X_{t_0}^m=X_{t_0}\). 

\emph{At an own opportunity.}
At \(\tau\in\mathcal G_m\), population \(m\) forms the candidate
\[
u_m^{\rm cand}
=\mathfrak R_m\bigl(\tau,\widehat X_\tau^m,
\widehat u_{n|m}^{\rm act}|_{[\tau,T]}\bigr).
\]
The candidate is compared in \(\mathbb L(\tau)\) with
\(u_m^{\rm act,-}|_{[\tau,T]}\). Equality is a pass; otherwise population
\(m\) adopts \(u_m^{\rm cand}\) and announces the genuine revision
\((\tau,m)\) for \(\mathsf{Rec}\) in \cref{eq:event-only-record}.

\emph{After an opponent revision.}
If \((\sigma,n)\) is announced, population \(m\) regenerates the new
opponent continuation plan using its reconstructed state and its own
active plan:
\begin{equation}\label{eq:event-record-regeneration}
\widehat u_{n|m}^{\rm new}
=\mathfrak R_n(\sigma,\widehat X_\sigma^m,u_m^{\rm act}|_{[\sigma,T]}).
\end{equation}
The stored plan is replaced on \([\sigma,T]\):
\[
\widehat u_{n|m}^{\rm act}|_{[\sigma,T]}
=\widehat u_{n|m}^{\rm new}.
\]
\subsection{Exact local implementation}
\begin{theorem}[Exact local implementation]
\label{thm:mean-field-exact-implementation}\label{thm:record-recursion}
Under
\cref{ass:lq,ass:eq-wellposed,ass:response-wellposed,ass:event-grid,ass:protocol-observability},
the local implementation coincides with the ideal asynchronous replanning
process on every finite chronological opportunity prefix. In particular,
for \(m\in\{1,2\}\) and \(n=3-m\),
\begin{equation}\label{eq:local-full-prefix-equality}
\widehat X_t^m=X_t,
\qquad
\widehat u_{n|m}^{\rm act}|_{[t,T]}
=
u_n^{\rm act}|_{[t,T]},
\end{equation}
between opportunities and immediately before and after every processed opportunity.
Hence the two processes have the same pass/revision decisions, public
revision record, active continuation plans, and physical aggregate path.
\end{theorem}

\begin{proof}
Induct over a finite chronological opportunity prefix. The invariant
is \cref{eq:local-full-prefix-equality} for both populations, together
with equality of physical states, active plans, and public records in
the local and ideal processes. It holds at \(t_0\) by
\cref{eq:protocol-initializer}, with the common plans \(u_m^0\) and empty
record.

Suppose it holds after time \(a\), before the next opportunity \(\tau\).
The common active plans generate the same physical state. The stored
opponent plans are exact, so \(r_i=\widehat X^i-X\) satisfies
\[
r_i(t)=\int_a^t(\mathsf A-L_i(s)H_i)r_i(s)\,ds,\qquad r_i(a)=0.
\]
Uniqueness gives \(r_i=0\) through \(\tau\). If the actor is \(m\),
both procedures therefore evaluate \(\mathfrak R_m\) on the same state
and opponent plan, compare with the same active plan, and make the
same decision. A pass preserves the invariant. At a genuine revision,
both adopt the same candidate and announce \((\tau,m)\). The non-actor
\(n\) regenerates
\[
\widehat u_{m|n}^{\rm new}
=\mathfrak R_m(\tau,\widehat X_\tau^n,u_n^{\rm act,-}|_{[\tau,T]})
=\mathfrak R_m(\tau,X_\tau,u_n^{\rm act,-}|_{[\tau,T]})
=u_m^{\rm act,+}|_{[\tau,T]}.
\]
The actor's stored opponent plan remains exact, and the physical state
is continuous. Thus the invariant holds after the update, closing the
induction.
\end{proof}

Local finiteness before \(\bar t\) makes every pre-accumulation interval
a finite prefix. Under \cref{thm:zeno-strategy-completion}, exact local
implementation therefore inherits bounded execution, the state left
limit, and strong plan convergence to
\[
\mathcal E_{\bar t}(X_{\bar t-}).
\]

\section{Finite-population robustness}\label{sec:finite-population-robustness}

\Cref{thm:mean-field-exact-implementation} uses exact mean-field
aggregates. Empirical aggregates fluctuate around these means, so exact
equality of candidate and active plans is not a robust pass criterion.
We introduce a vanishing tolerance in the finite-population implementation
and prove record agreement and mean-square error bounds on each fixed
finite chronological opportunity prefix.

Let population \(m\) contain \(N_m\) agents using
\(u_m=\bar u^m=u_m^{N,\rm act}\) in \cref{eq:model-state}, with
aggregate input \(X^N\). Set
\[
N_*:=\min(N_1,N_2),\qquad
z_t^{m,N}:=N_m^{-1}\sum_{i=1}^{N_m}x_t^{m,i,N},\qquad
X_t^N:=\bigl((z_t^{1,N})^\top,(z_t^{2,N})^\top\bigr)^\top.
\]
Averaging the individual dynamics gives
\begin{equation}\label{eq:finite-aggregate-dynamics}
dX_t^N=\bigl(\mathsf AX_t^N+\mathsf B_1u_1^{N,\rm act}(t)
+\mathsf B_2u_2^{N,\rm act}(t)\bigr)dt+dM_t^N,
\end{equation}
where the population-\(m\) block of \(M_t^N\) is
\(N_m^{-1}\sum_iD_mW_t^{m,i}\).

\begin{assumption}[Finite-population sampling]
\label{ass:finite-population-sampling}
Initial states are independent within each population, have means
\(z_0^m\), and have uniformly bounded second moments. The Brownian
motions are mutually independent and independent of the initial states.
Both systems use the same prescribed plans \(u_1^0,u_2^0\) on
\([0,t_0]\) and the same fixed exogenous opportunity grids.
\end{assumption}
In particular,
\begin{equation}\label{eq:finite-sampling-bound}
\mathbb E|X_0^N-X_0|^2+\mathbb E\sup_{t\le T}|M_t^N|^2\le C/N_*.
\end{equation}
Population \(m\) observes \(z^{m,N}=H_mX^N\) and applies
\cref{eq:protocol-initializer} to the empirical residual
\(y_m^N:=z^{m,N}-H_mX^{m,\mathrm{base}}\), keeping the same known
base trajectory and law means as in \cref{eq:base-trajectory-map}.
After \(t_0\), population \(m\) propagates \(\widehat X^{m,N}\) by
\cref{eq:causal-integral-observer} with \(z^{m,N}\) and the finite-population
active and stored plans, then applies the response and regeneration
rules of \cref{subsec:causal-propagation} to the reconstructed quantities.

The bounded initializer acts
pathwise on \(y_m^N\in\mathcal Y_m\); the causal integral observer and
continuous affine response maps then define measurable, adapted plans
on each finite prefix.

For a tolerance satisfying
\begin{equation}\label{eq:finite-threshold-scaling}
\eta_N\longrightarrow0,\qquad N_*\eta_N^2\longrightarrow\infty,
\end{equation}
the finite-population decision at \(\tau\in\mathcal G_m\) is
\begin{equation}\label{eq:finite-threshold-rule}
\text{pass if }\ 
\|u_m^{N,\rm cand}-u_m^{N,\rm act,-}|_{[\tau,T]}\|_{\mathbb L(\tau)}
\le\eta_N;\qquad \text{revise otherwise}.
\end{equation}
The mean-field rule remains exact equality.

Fix \(K\ge1\) and a finite chronological opportunity prefix
\(t_0<\tau_1<\cdots<\tau_K<T\), with actor \(m_j\) at \(\tau_j\).
Write \(\gamma_j^{\rm MF}:=\|u_{m_j}^{\rm cand}
-u_{m_j}^{\rm act,-}|_{[\tau_j,T]}\|_{\mathbb L(\tau_j)}\) and
\(\gamma_j^N:=\|u_{m_j}^{N,\rm cand}
-u_{m_j}^{N,\rm act,-}|_{[\tau_j,T]}\|_{\mathbb L(\tau_j)}\).
The mean-field rule passes exactly when \(\gamma_j^{\rm MF}=0\) and
revises exactly when \(\gamma_j^{\rm MF}>0\); the finite-population
rule passes when \(\gamma_j^N\le\eta_N\) and revises otherwise. Set
\[
I_K^{\rm rev}:=\{j\in\{1,\ldots,K\}:\gamma_j^{\rm MF}>0\},\qquad
\delta_K:=\begin{cases}
\min_{j\in I_K^{\rm rev}}\gamma_j^{\rm MF},&I_K^{\rm rev}\ne\varnothing,\\
+\infty,&I_K^{\rm rev}=\varnothing.
\end{cases}
\]
The positive margin is automatic on a fixed prefix. Let \(\mathcal A_j^N\)
be the event that the finite-population and mean-field processes make
the same pass/revision decision at each of the first \(j\)
opportunities, with \(\mathcal A_0^N=\Omega\).
For \(\tau_0=t_0\), denote the post-opportunity active and stored plans
by \(u_m^{N,j},\widehat u_{n|m}^{N,j}\), and the mean-field plans by
\(u_m^j\). Define the prefix error, with \(n=3-m\), by
\[
\begin{aligned}
\mathrm{Err}_K^N:=\max\bigl\{
&\sup_{0\le t\le\tau_K}|X_t^N-X_t|,
\ \max_m\sup_{t_0\le t\le\tau_K}|\widehat X_t^{m,N}-X_t|,\\
&\max_{\substack{0\le j\le K\\m=1,2}}
\|u_m^{N,j}-u_m^j\|_{\mathbb L(\tau_j)},
\ \max_{\substack{0\le j\le K\\m=1,2}}
\|\widehat u_{n|m}^{N,j}-u_n^j\|_{\mathbb L(\tau_j)}\bigr\}.
\end{aligned}
\]
The four terms measure empirical aggregate, local state-reconstruction,
active-plan, and locally stored opponent-plan errors, respectively.

\begin{theorem}[Finite-population robustness on a fixed prefix]
\label{thm:finite-prefix-robustness}
Assume \cref{ass:lq,ass:eq-wellposed,ass:response-wellposed,ass:event-grid,ass:protocol-observability,ass:finite-population-sampling}
and \cref{eq:finite-threshold-scaling}. For every fixed finite
chronological opportunity prefix, a constant \(C_K\), independent of
\(N_1,N_2\), satisfies, for all sufficiently large \(N_*\),
\begin{align}
\mathbb P\bigl((\mathcal A_K^N)^c\bigr)
&\le C_K/(N_*\eta_N^2),\label{eq:finite-record-probability}\\
\mathbb E\bigl[\mathbf1_{\mathcal A_K^N}(\mathrm{Err}_K^N)^2\bigr]
&\le C_K/N_* .\label{eq:finite-prefix-error}
\end{align}
Thus the revision records agree with probability tending to one, and
the mean-square implementation error on the matching prefix is
\(O(N_*^{-1})\).
\end{theorem}
\begin{proof}
\emph{1. Initialization.}
Put \(\Delta^N=X^N-X\). On \([0,t_0]\),
\(d\Delta^N=\mathsf A\Delta^Ndt+dM^N\).
The sampling bound, BDG, and Gronwall give
\(\mathbb E\sup_{t\le t_0}|\Delta_t^N|^2\le C/N_*\).
Since \(y_m^N-y_m=H_m\Delta^N\), the perturbation bound in
\cref{thm:protocol-sufficient-observation} gives the same order for
the initialized state and stored-plan errors; the active-plan errors
are initially zero.

\emph{2. Inter-event propagation.}
On \([\tau_{j-1},\tau_j]\), set
\(r_m^N=\widehat X^{m,N}-X\),
\(a_m^N=u_m^{N,\rm act}-u_m^{\rm act}\), and
\(d_{n|m}^N=\widehat u_{n|m}^{N,\rm act}-u_n^{\rm act}\).
No plan is updated between \(\tau_{j-1}\) and \(\tau_j\). On
\(\mathcal A_{j-1}^N\), their errors on this interval are controlled
by post-opportunity errors at \(\tau_{j-1}\), since restriction
cannot increase an \(L^2\) norm. Before the next update,
\begin{align}
d\Delta^N&=(\mathsf A\Delta^N+\mathsf B_1a_1^N
+\mathsf B_2a_2^N)dt+dM^N,\label{eq:finite-aggregate-error}\\
\dot r_m^N&=(\mathsf A-L_mH_m)r_m^N+\mathsf B_ma_m^N
+\mathsf B_nd_{n|m}^N+L_mH_m\Delta^N.
\label{eq:finite-observer-error}
\end{align}
Bounded coefficients and BDG/Gronwall preserve the \(O(N_*^{-1})\)
mean-square bound over this fixed interval, restricted to
\(\mathcal A_{j-1}^N\). More explicitly, the pathwise integral
estimate bounds each state supremum by the preceding state errors,
the \(L^2\) plan errors, and \(C\sup_{t\le T}|M_t^N|\).
Multiplication by \(\mathbf1_{\mathcal A_{j-1}^N}\) and
\cref{eq:finite-sampling-bound} therefore require no martingale
property conditional on record agreement.

\emph{3. Response and regeneration.}
The dependence estimate \cref{eq:response-uniform-w1infty-increment}
applies pathwise to random inputs and implies
\[
\|u_m^{N,\rm cand}-u_m^{\rm cand}\|_{\mathbb L(\tau_j)}
\le C\bigl(|r_m^N(\tau_j)|
+\|d_{n|m}^N\|_{\mathbb L(\tau_j)}\bigr).
\]
Squaring the candidate bound, multiplying by
\(\mathbf1_{\mathcal A_{j-1}^N}\), and taking expectations, the state
term is controlled by the preceding BDG/Gronwall estimate from
\cref{eq:finite-aggregate-error,eq:finite-observer-error},
and the stored-plan term by the induction hypothesis and restriction.
At a common revision, the actor adopts its candidate and the non-actor
regenerates the actor's plan; the latter bound uses
\((r_n^N,a_n^N)\) in the same response estimate. A common pass only
restricts the existing plans. Since
\(\mathcal A_j^N\subseteq\mathcal A_{j-1}^N\), these pathwise bounds
close the induction for errors multiplied by
\(\mathbf1_{\mathcal A_j^N}\). Since \(K\) is fixed, the maximum
over finitely many intervals and updates changes only the constant,
proving \cref{eq:finite-prefix-error}.
For the finite-population gap defined above, the reverse triangle
inequality gives the pre-decision estimate
\begin{equation}\label{eq:finite-gap-error}
\mathbb E\bigl[\mathbf1_{\mathcal A_{j-1}^N}
|\gamma_j^N-\gamma_j^{\rm MF}|^2\bigr]\le C_K/N_*.
\end{equation}

\emph{4. Decision consistency.}
The complement of \(\mathcal A_K^N\) is the union over \(j\le K\) of
\(\{\mathcal A_{j-1}^N\text{ holds and the }j\text{-th decision mismatches}\}\).
At a mean-field pass, \(\gamma_j^{\rm MF}=0\), so Chebyshev bounds the
probability of a first mismatch at \(j\) by
\(C_K/(N_*\eta_N^2)\). At a genuine revision,
\(\gamma_j^{\rm MF}\ge\delta_K\); once \(\eta_N<\delta_K/2\), a mismatch
requires \(|\gamma_j^N-\gamma_j^{\rm MF}|\ge\delta_K/2\), with first-mismatch
probability at most \(4C_K/(N_*\delta_K^2)\).
Summing over \(j\le K\), and taking \(\eta_N\le1\), proves
\cref{eq:finite-record-probability}.
\end{proof}

\begin{remark}\label{rem:finite-prefix-only}\label{rem:no-finite-nash}
The constant \(C_K\) is not claimed uniform as \(K\to\infty\), so
the theorem does not give a finite-population approximation through
Zeno accumulation. It concerns implementation robustness and makes
no \(\varepsilon\)-Nash claim for the revised finite-player process.
\end{remark}

\section{Conclusion}\label{sec:conclusion}

Asynchronous continuation replanning requires control of both the
information used in a response and the interaction of successive
responses. Local observations need only determine the state--plan target
that initializes the protocol. Stability, however, concerns the full
remaining plans as well as the portions physically executed. The scalar
counterexample shows why a convergent state cannot settle this second
question. Spectral stability of the limiting response cycle controls
both objects at a pre-terminal accumulation. The limiting equilibrium
starts from the state actually reached, which need not lie on the
original complete-information trajectory.

The exact local implementation connects this analysis to the available
observations and public revision record. With empirical observations,
the tolerance in the pass test addresses errors in discrete decisions
as well as plan reconstruction. The resulting robustness estimates apply
on fixed finite chronological opportunity prefixes; they are not uniform
through Zeno accumulation and give no \(\varepsilon\)-Nash result for the
revised finite-player process. Two questions remain: how to implement the
protocol from noisy or sampled observations with controlled reconstruction
and decision errors, and what additional conditions can ensure convergence
at the spectral boundary \(r(Q_Z)=1\).

\label{end:core}

\appendix
\section{Baseline Open-Loop LQ--MFG Derivation}\label{app:baseline-lq}
Set \(\mathcal A=\operatorname{diag}(A_1,A_2)\), and analogously define
\(\mathcal B,\mathcal F,\mathcal R\); let
\(\mathcal C=(C_1^\top,C_2^\top)^\top\). The stacked cost blocks are

\begin{equation}\label{eq:block-Q-definitions}
\begin{aligned}
\mathcal Q_{m\ell}
&:=\delta_{m\ell}\left(Q_{Im}+\sum_{r=1}^2Q_m^r\right)
-Q_m^\ell\Gamma_m^\ell,\\
\bar{\mathcal Q}_{m\ell}
&:=\delta_{m\ell}\left(\bar Q_{Im}+\sum_{r=1}^2\bar Q_m^r\right)
-\bar Q_m^\ell\bar\Gamma_m^\ell,
\end{aligned}
\end{equation}

\begin{equation}\label{eq:block-nu-definitions}
\nu_m:=-Q_{Im}s_m-\sum_{r=1}^2Q_m^r\eta_m^r,
\qquad
\bar\nu_m:=-\bar Q_{Im}\bar s_m
-\sum_{r=1}^2\bar Q_m^r\bar\eta_m^r,
\end{equation}

with \(\mathcal Q=(\mathcal Q_{m\ell})\),
\(\bar{\mathcal Q}=(\bar{\mathcal Q}_{m\ell})\),
\(\nu=(\nu_1^\top,\nu_2^\top)^\top\), and
\(\bar\nu=(\bar\nu_1^\top,\bar\nu_2^\top)^\top\).

\begin{proposition}[Global Riccati chart and continuation boundary problems]
\label{prop:global-no-conjugate}
Under \cref{ass:eq-wellposed}, \(P=VU^{-1}\) is the unique solution of
\begin{equation}\label{eq:global-riccati}
\left\{
\begin{aligned}
-\dot P_t
&=P_t(\mathcal A+\mathcal C)+\mathcal A^\top P_t+\mathcal Q
-P_t(\mathcal B+\mathcal F)\mathcal R^{-1}\mathcal B^\top P_t,\\
P_T&=\bar{\mathcal Q}.
\end{aligned}
\right.
\end{equation}
For every \(\tau\in[0,T]\), the global continuation boundary problem
\cref{eq:equilibrium-system} on \([\tau,T]\) has a unique solution for
every boundary state and affine forcing.
If \(U(\tau)\) is singular, the homogeneous problem has a nonzero solution
with zero state at \(\tau\), and some inhomogeneous boundary data are
unattainable.
\end{proposition}

\begin{proof}
Differentiating \(VU^{-1}\) gives \cref{eq:global-riccati}. Conversely,
\[
\dot U=(\mathcal A+\mathcal C-(\mathcal B+\mathcal F)
\mathcal R^{-1}\mathcal B^\top P)U,\qquad U(T)=I,\qquad V=PU
\]
recover the lift. Riccati uniqueness follows from local Lipschitz
continuity. Terminal graph data \(\xi\) give state \(U(\tau)\xi\)
at \(\tau\); forcing adds a fixed translation. Thus invertibility
is exactly the boundary solvability criterion, while singularity gives
both a nonzero zero-boundary solution and unattainable data.
\end{proof}

\begin{theorem}[Baseline deterministic open-loop equilibrium]
\label{thm:mfe}
Under \cref{ass:lq,ass:eq-wellposed}, the complete-information deterministic
open-loop mean-field equilibrium is unique.  With
\(Y=((p^1)^\top,(p^2)^\top)^\top\), it solves
\begin{equation}\label{eq:equilibrium-system}
\left\{
\begin{aligned}
\frac d{dt}\begin{pmatrix}X_t\\Y_t\end{pmatrix}
&=
\begin{pmatrix}
\mathcal A+\mathcal C&-(\mathcal B+\mathcal F)\mathcal R^{-1}\mathcal B^\top\\
-\mathcal Q&-\mathcal A^\top
\end{pmatrix}
\begin{pmatrix}X_t\\Y_t\end{pmatrix}
-\begin{pmatrix}0\\\nu\end{pmatrix},\\
X_0&=X_0,
\qquad Y_T=\bar{\mathcal Q}X_T+\bar\nu,
\end{aligned}
\right.
\end{equation}
with \(u_m(t)=-R_m^{-1}B_m^\top p_t^m\).
Writing \(Y_t=P_tX_t+\mathcal G_t\), the affine term is
determined by
\begin{equation}\label{eq:global-G}
\dot{\mathcal G}_t
=-\bigl(\mathcal A^\top
-P_t(\mathcal B+\mathcal F)\mathcal R^{-1}\mathcal B^\top\bigr)
\mathcal G_t-\nu,
\qquad \mathcal G_T=\bar\nu.
\end{equation}
\end{theorem}

Set \(A_P=\mathcal A+\mathcal C-(\mathcal B+\mathcal F)
\mathcal R^{-1}\mathcal B^\top P\), with transition \(\Phi(t,s)\)
and \(\Phi_t=\Phi(t,0)\). Stacking the controls as \(u\), the affine solution map
\(\mathcal S:X_0\mapsto(X,Y,u_1,u_2)\) is given by
\begin{align}
\dot X_t&=A_P(t)X_t
-(\mathcal B+\mathcal F)\mathcal R^{-1}\mathcal B^\top\mathcal G_t,
\qquad X_0\ \text{given},\label{eq:closed-baseline-forward}\\
Y_t&=P_tX_t+\mathcal G_t,\qquad
u_t=-\mathcal R^{-1}\mathcal B^\top(P_tX_t+\mathcal G_t).
\label{eq:baseline-plan-formula}
\end{align}
For an initial increment \(H\in\mathbb R^{2d}\), subtraction gives
\begin{equation}\label{eq:solution-map-sensitivity}
\begin{aligned}
X_t[X_0+H]-X_t[X_0]&=\Phi_tH,\\
Y_t[X_0+H]-Y_t[X_0]&=P_t\Phi_tH,\\
u_t[X_0+H]-u_t[X_0]
&=-\mathcal R^{-1}\mathcal B^\top P_t\Phi_tH.
\end{aligned}
\end{equation}

For fixed deterministic aggregates, put \(\widehat x^m=\mathbb E x^m\).
The representative adjoint is
\begin{equation}\label{eq:agent-adjoint}
\left\{
\begin{aligned}
dp_t^m
&=-\left[A_m^\top p_t^m+Q_{Im}(\widehat x_t^m-s_m)
+\sum_{n=1}^2Q_m^n(\widehat x_t^m-\psi_m^n(X_t))\right]dt,\\
p_T^m
&=\bar Q_{Im}(\widehat x_T^m-\bar s_m)
+\sum_{n=1}^2\bar Q_m^n
(\widehat x_T^m-\bar\psi_m^n(X_T)).
\end{aligned}
\right.
\end{equation}

\begin{theorem}[Deterministic open-loop LQ optimal control]
\label{thm:open-loop-optimal}
Under \cref{ass:lq}, for each deterministic \((X,\bar u^m)\), the
representative problem has a unique optimizer in
\(\mathcal U_m^{\mathrm{det}}\), characterized by
\(u_m^*(t)=-R_m^{-1}B_m^\top p_t^m\), where \(p^m\) solves \cref{eq:agent-adjoint} along the optimal mean state.
\end{theorem}

\begin{proof}
The mean equation is
\begin{equation}\label{eq:proof-mean-state}
\dot{\widehat x}=A_m\widehat x+B_mu+C_mX+F_m\bar u^m,
\qquad \widehat x_0=\mathbb E x_0^m.
\end{equation}
The centered state solves \(d(x^m-\widehat x)=A_m(x^m-\widehat x)dt+D_m dW^m\),
so its contribution to the cost is independent of deterministic \(u\).
The reduced quadratic functional on \(L^2\) has a bounded Hessian
bounded below by \(\lambda_{\min}(R_m)I\); hence it has a unique
minimizer. Integration by parts with \cref{eq:agent-adjoint} gives
\(DJ_m^{\rm MF}(u)h=\int_0^T\langle R_mu+B_m^\top p^m,h\rangle dt\),
which proves the stated characterization.
\end{proof}

\begin{proof}[Proof of \cref{thm:mfe}]
Substitute \(\widehat x^m=z^m\), \(\bar u^m=u_m\), and
\(u_m=-R_m^{-1}B_m^\top p^m\) in the mean and adjoint equations.
Stacking gives \cref{eq:equilibrium-system}; the converse follows
from \cref{thm:open-loop-optimal}. The boundary criterion gives
unique solvability, and substitution of \(Y=PX+\mathcal G\)
gives \cref{eq:global-G,eq:closed-baseline-forward,eq:baseline-plan-formula}.
\end{proof}

\section{Local Frozen-Response Derivation}
\label{app:local-frozen-details}

\subsection{Frozen boundary system and local Riccati chart}
Fix \(m\), put \(n=3-m\), and use actor-first coordinates
\(X^{[m]}=(z^m,z^n)^\top=\mathfrak P_mX\), with boundary
\(x^{[m]}=\mathfrak P_mx\), where
\begin{equation}\label{eq:global-to-local-permutation}
\mathfrak P_1=I_{2d},\qquad
\mathfrak P_2=\begin{pmatrix}0&I_d\\I_d&0\end{pmatrix}.
\end{equation}
The local blocks are
\begin{equation}\label{eq:local-frozen-coefficients}
\begin{aligned}
\mathsf A_m^{\mathrm{fr}}
&:=
\begin{pmatrix}
A_m+C_m^m&C_m^n\\
C_n^m&A_n+C_n^n
\end{pmatrix},\\
\mathsf D_m^{\mathrm{fr}}
&:=
\begin{pmatrix}
(B_m+F_m)R_m^{-1}B_m^\top\\
0
\end{pmatrix},
&
\mathsf E_m^{\mathrm{fr}}
&:=
\begin{pmatrix}
0\\B_n+F_n
\end{pmatrix},\\
\mathsf Q_m^{\mathrm{fr}}
&:=
\begin{pmatrix}
Q_{Im}+\sum_{r=1}^2Q_m^r-Q_m^m\Gamma_m^m&
-Q_m^n\Gamma_m^n
\end{pmatrix},\\
\bar{\mathsf Q}_m^{\mathrm{fr}}
&:=
\begin{pmatrix}
\bar Q_{Im}+\sum_{r=1}^2\bar Q_m^r
-\bar Q_m^m\bar\Gamma_m^m&
-\bar Q_m^n\bar\Gamma_m^n
\end{pmatrix}.
\end{aligned}
\end{equation}
\begin{proposition}[Representative optimality and population consistency]
\label{prop:local-frozen-optimality}
Under \cref{ass:lq}, fix \(\tau\in[t_0,T]\),
\(x^{[m]}\in\mathbb R^{2d}\), and \(v\in\mathbb L(\tau)\).
A pair \((X^{[m]},u_m)\in H^1\times L^2\) is representative-agent
optimal and population-\(m\) consistent with frozen opponent plan
\(v\) exactly when \(u_m=-R_m^{-1}B_m^\top y^m\) and
\begin{equation}\label{eq:frozen-fbs}
\left\{
\begin{aligned}
\dot X_t^{[m]}
&=\mathsf A_m^{\mathrm{fr}}X_t^{[m]}
-\mathsf D_m^{\mathrm{fr}}y_t^m
+\mathsf E_m^{\mathrm{fr}}v_t,\\
-\dot y_t^m
&=A_m^\top y_t^m
+\mathsf Q_m^{\mathrm{fr}}X_t^{[m]}+\nu_m,\\
X_\tau^{[m]}&=x^{[m]},\\
y_T^m&=\bar{\mathsf Q}_m^{\mathrm{fr}}X_T^{[m]}+\bar\nu_m.
\end{aligned}
\right.
\end{equation}
Thus \(\widetilde{\mathfrak R}_m(\tau,x^{[m]},v)=-R_m^{-1}B_m^\top y^m\).
\end{proposition}

\begin{proof}
Representative stationarity in \cref{thm:open-loop-optimal} gives
\(u_m=-R_m^{-1}B_m^\top y^m\), with aggregate paths held fixed.
Imposing consistency makes the cost gradients
\(\mathsf Q_m^{\rm fr}X^{[m]}+\nu_m\) and
\(\bar{\mathsf Q}_m^{\rm fr}X_T^{[m]}+\bar\nu_m\).
The mean dynamics use \((B_m+F_m)u_m\) and \((B_n+F_n)v\),
giving \cref{eq:frozen-fbs}. Strict convexity proves the converse.
\end{proof}

Let
\begin{equation}\label{eq:local-frozen-lift}
\frac d{dt}
\begin{pmatrix}U_m\\V_m\end{pmatrix}
=
\begin{pmatrix}
\mathsf A_m^{\mathrm{fr}}&-\mathsf D_m^{\mathrm{fr}}\\
-\mathsf Q_m^{\mathrm{fr}}&-A_m^\top
\end{pmatrix}
\begin{pmatrix}U_m\\V_m\end{pmatrix},
\qquad
\begin{pmatrix}U_m(T)\\V_m(T)\end{pmatrix}
=
\begin{pmatrix}I_{2d}\\\bar{\mathsf Q}_m^{\mathrm{fr}}\end{pmatrix},
\end{equation}
\begin{proposition}[Local Riccati chart and boundary criterion]
\label{prop:frozen-no-conjugate}\label{prop:local-frozen-riccati}
Fix \([a,T]\subset[t_0,T]\).  The following statements are equivalent:
\begin{enumerate}[label=\textup{(\roman*)}]
\item \(U_m(t)\) is invertible for every \(t\in[a,T]\);
\item the rectangular Riccati equation
\begin{equation}\label{eq:local-frozen-riccati}
\left\{
\begin{aligned}
-\dot\Pi_m^{\mathrm{fr}}
&=A_m^\top\Pi_m^{\mathrm{fr}}
+\Pi_m^{\mathrm{fr}}\mathsf A_m^{\mathrm{fr}}
-\Pi_m^{\mathrm{fr}}\mathsf D_m^{\mathrm{fr}}
\Pi_m^{\mathrm{fr}}
+\mathsf Q_m^{\mathrm{fr}},\\
\Pi_m^{\mathrm{fr}}(T)&=\bar{\mathsf Q}_m^{\mathrm{fr}}
\end{aligned}
\right.
\end{equation}
has a solution on \([a,T]\);
\item for every \(\tau\in[a,T]\), initial state \(x\), and opponent input
\(v\in L^2([\tau,T];\R^{d_0})\), the boundary system
\cref{eq:frozen-fbs} has a unique solution.
\end{enumerate}
Under these conditions,
\[
\Pi_m^{\mathrm{fr}}(t)=V_m(t)U_m(t)^{-1}.
\]
For a fixed input \(v\), the affine adjoint component solves
\begin{equation}\label{eq:local-frozen-affine}
\left\{
\begin{aligned}
-\dot g_m^v
&=\bigl(A_m^\top-\Pi_m^{\mathrm{fr}}
\mathsf D_m^{\mathrm{fr}}\bigr)g_m^v
+\Pi_m^{\mathrm{fr}}\mathsf E_m^{\mathrm{fr}}v+\nu_m,\\
g_m^v(T)&=\bar\nu_m,
\end{aligned}
\right.
\end{equation}
and
\begin{equation}\label{eq:local-frozen-decoupling}
y_t^m=\Pi_m^{\mathrm{fr}}(t)X_t^{[m]}+g_m^v(t).
\end{equation}
The closed state equation is
\begin{equation}\label{eq:local-frozen-closed-state}
\dot X_t^{[m]}
=\bigl(\mathsf A_m^{\mathrm{fr}}
-\mathsf D_m^{\mathrm{fr}}\Pi_m^{\mathrm{fr}}(t)\bigr)X_t^{[m]}
-\mathsf D_m^{\mathrm{fr}}g_m^v(t)
+\mathsf E_m^{\mathrm{fr}}v_t.
\end{equation}
\end{proposition}

\begin{proof}
As in \cref{prop:global-no-conjugate}, differentiation of
\(V_mU_m^{-1}\) gives the Riccati equation. Conversely,
\(\dot U_m=(\mathsf A_m^{\rm fr}-\mathsf D_m^{\rm fr}\Pi_m^{\rm fr})U_m\),
\(U_m(T)=I\), and \(V_m=\Pi_m^{\rm fr}U_m\) recover the lift.
Homogeneous terminal graph data \(\xi\) give state \(U_m(\tau)\xi\);
\(v,\nu_m,\bar\nu_m\) add a fixed translation. This proves all
three equivalences, including failure of uniqueness when \(U_m(\tau)\)
is singular. Substitution of \cref{eq:local-frozen-decoupling}
gives the affine and closed state equations.
\end{proof}

The opponent input is exogenous, so this local chart is independent
of the global equilibrium chart.

\subsection{Affine response formulas}

Set \(\mathsf F_m=A_m^\top-\Pi_m^{\rm fr}\mathsf D_m^{\rm fr}\)
and \(\mathsf A_m^{\rm cl}=\mathsf A_m^{\rm fr}-\mathsf D_m^{\rm fr}\Pi_m^{\rm fr}\).
Let \(\Lambda_m(t,s)\), \(t\le s\), and
\(\Phi_m^{\mathrm{cl}}(t,s)\), \(s\le t\), be defined by
\[
\partial_t\Lambda_m(t,s)=-\mathsf F_m(t)\Lambda_m(t,s),
\quad \Lambda_m(s,s)=I_d,
\]
\[
\partial_t\Phi_m^{\mathrm{cl}}(t,s)
=\mathsf A_m^{\mathrm{cl}}(t)\Phi_m^{\mathrm{cl}}(t,s),
\quad \Phi_m^{\mathrm{cl}}(s,s)=I_{2d}.
\]
For \(t\in[\tau,T]\), set
\begin{align}
g_m^0(t)
&:=\Lambda_m(t,T)\bar\nu_m
+\int_t^T\Lambda_m(t,s)\nu_m\,ds,
\label{eq:local-frozen-g0}\\
(\mathcal L_m^\tau v)(t)
&:=\int_t^T\Lambda_m(t,s)\Pi_m^{\mathrm{fr}}(s)
\mathsf E_m^{\mathrm{fr}}v(s)\,ds,
\label{eq:local-frozen-input-adjoint}\\
X_m^x(t)
&:=\Phi_m^{\mathrm{cl}}(t,\tau)x^{[m]},
\label{eq:local-frozen-state-part}\\
X_m^v(t)
&:=\int_\tau^t\Phi_m^{\mathrm{cl}}(t,s)
\left[\mathsf E_m^{\mathrm{fr}}v(s)
-\mathsf D_m^{\mathrm{fr}}(\mathcal L_m^\tau v)(s)\right]ds,
\label{eq:local-frozen-input-state}\\
X_m^0(t)
&:=-\int_\tau^t\Phi_m^{\mathrm{cl}}(t,s)
\mathsf D_m^{\mathrm{fr}}g_m^0(s)\,ds.
\label{eq:local-frozen-affine-state}
\end{align}
The locally ordered state gain and the opponent-input and affine terms are
\begin{equation}\label{eq:explicit-response-gains}
\begin{aligned}
(\widetilde K_m^\tau x^{[m]})(t)
&=-R_m^{-1}B_m^\top\Pi_m^{\mathrm{fr}}(t)X_m^x(t),\\
(\mathcal J_m^\tau v)(t)
&=-R_m^{-1}B_m^\top
\left[\Pi_m^{\mathrm{fr}}(t)X_m^v(t)
+(\mathcal L_m^\tau v)(t)\right],\\
h_m^\tau(t)
&=-R_m^{-1}B_m^\top
\left[\Pi_m^{\mathrm{fr}}(t)X_m^0(t)+g_m^0(t)\right].
\end{aligned}
\end{equation}

\begin{proof}[Proof of \cref{thm:frozen-response}]
Variation of constants gives \(g_m^v=g_m^0+\mathcal L_m^\tau v\)
and \(X^{[m]}=X_m^x+X_m^v+X_m^0\). The control law yields
\cref{eq:explicit-response-gains} and
\(K_m^\tau=\widetilde K_m^\tau\mathfrak P_m\).
The uniform bounds follow from the next lemma and its difference
estimate, using \(\|f\|_2\le\sqrt{T-t_0}\|f\|_\infty\).
\end{proof}

\begin{lemma}[Uniform regularity of continuation responses]
\label{lem:response-uniform-w1infty}
Under \cref{ass:lq,ass:response-wellposed}, there is a constant
\(C_{\mathrm{reg}}<\infty\), depending only on the fixed model data and
\([t_0,T]\), such that, for \(m=1,2\), \(t_0\le\tau<T\),
\(x\in\mathbb R^{2d}\), and \(v\in\mathbb L(\tau)\), the response has a
\(C^1\) representative and
\begin{equation}\label{eq:response-uniform-w1infty}
 \|\mathfrak R_m(\tau,x,v)\|_{W^{1,\infty}([\tau,T])}
 \le C_{\mathrm{reg}}\bigl(1+|x|+\|v\|_{\mathbb L(\tau)}\bigr).
\end{equation}
Here \(\|f\|_{W^{1,\infty}}=\|f\|_\infty+\|f'\|_\infty\).
All norms are on the native continuation interval. At \(\tau=T\),
the zero-length plan convention applies.
\end{lemma}
\begin{proof}
The local chart and both transition matrices are uniformly bounded on
their compact time domains. The displayed formulas and
\(\|v\|_1\le\sqrt{T-t_0}\|v\|_2\) give, uniformly in \(\tau\),
\[
\|g_m^0\|_\infty\le C,\quad
\|\mathcal L_m^\tau v\|_\infty+\|X_m^v\|_\infty\le C\|v\|_2,
\quad \|X_m^x\|_\infty+\|X_m^0\|_\infty\le C(1+|x|).
\]
Hence \(\|X^{[m]}\|_\infty+\|y^m\|_\infty
\le C(1+|x|+\|v\|_2)\). The state has an \(L^2\) derivative,
so it is continuous; the adjoint equation in \cref{eq:frozen-fbs}
then gives a continuous derivative and
\[
\dot u_m=R_m^{-1}B_m^\top
(A_m^\top y^m+\mathsf Q_m^{\rm fr}X^{[m]}+\nu_m).
\]
This proves the \(C^1\) assertion and the uniform derivative bound,
without differentiating \(v\) or using a shrinking-interval
Sobolev constant.
\end{proof}

Subtracting two copies of the frozen boundary system
\cref{eq:frozen-fbs} at the same boundary removes the affine
offsets and also gives the uniform dependence estimate
\begin{equation}\label{eq:response-uniform-w1infty-increment}
\begin{aligned}
 &\|\mathfrak R_m(\tau,x,v)-\mathfrak R_m(\tau,x',v')\|_{W^{1,\infty}([\tau,T])}\le C_{\mathrm{reg}}
       \bigl(|x-x'|+\|v-v'\|_{\mathbb L(\tau)}\bigr).
\end{aligned}
\end{equation}

\subsection{Compactness and a coefficient-level gain bound}

\begin{proposition}[Response compactness]
\label{prop:response-operator-compact}
Under \cref{ass:lq,ass:response-wellposed}, for every
\(m\in\{1,2\}\) and \(\tau\in[t_0,T]\),
\[
\mathcal J_m^\tau:\mathbb L(\tau)\longrightarrow\mathbb L(\tau)
\]
is Hilbert--Schmidt and hence compact.  There is a coefficient-dependent
constant \(C_m^{\mathrm{HS}}<\infty\), uniform in \(\tau\), such that
\begin{equation}\label{eq:response-short-horizon-bound}
\|\mathcal J_m^\tau\|
\le\|\mathcal J_m^\tau\|_{\mathrm{HS}}
\le C_m^{\mathrm{HS}}\|\mathsf E_m^{\mathrm{fr}}\|(T-\tau).
\end{equation}
Consequently
\begin{equation}\label{eq:cycle-short-horizon-bound}
r(\mathcal J_q^\tau\mathcal J_p^\tau)
\le C_p^{\mathrm{HS}}C_q^{\mathrm{HS}}
\|\mathsf E_p^{\mathrm{fr}}\|
\|\mathsf E_q^{\mathrm{fr}}\|(T-\tau)^2.
\end{equation}
Thus sufficiently short remaining horizons are spectrally stable.
Uniform conclusions for parameter families additionally require uniform
response-chart bounds.
\end{proposition}

\begin{proof}
Substitution of \cref{eq:local-frozen-input-adjoint,eq:local-frozen-input-state}
in \cref{eq:explicit-response-gains} and Fubini give the kernel
\(j_\tau\) displayed in \cref{eq:extended-response-kernel}.
Fubini is valid since the coefficients are bounded and \(L^2\subset L^1\)
on the finite interval. The local chart and transition bounds give
\(\|j_\tau(t,r)\|_{\rm F}\le C_m^{\mathrm{HS}}\|\mathsf E_m^{\rm fr}\|\),
uniformly on \([\tau,T]^2\). Integration on that square proves
\cref{eq:response-short-horizon-bound}; likewise,
\begin{equation}\label{eq:adjoint-hs-bound}
\|\mathcal L_m^\tau\|_{\rm HS}
\le C\|\mathsf E_m^{\rm fr}\|(T-\tau).
\end{equation}
Hilbert--Schmidt operators are compact, and
\(r(AB)\le\|A\|\|B\|\) gives \cref{eq:cycle-short-horizon-bound}.
\end{proof}

\section{Continuity of the response at a moving boundary}
\label{app:boundary-continuity}

\begin{proposition}[Norm continuity of zero-extended LQ responses]
\label{prop:response-boundary-continuity}
Under \cref{ass:lq,ass:response-wellposed}, as \(\tau\) varies on
\([t_0,T]\), the maps \(\widehat K_m^\tau,\widehat J_m^\tau\) and the
affine term \(\widehat h_m^\tau\) are continuous in
\(\mathcal L(\mathbb R^{2d},\mathcal H)\), the Hilbert--Schmidt class on
\(\mathcal H\), and \(\mathcal H\), respectively. In particular,
\(\widehat J_m^{\tau_k}\to\widehat J_m^\tau\) in operator norm whenever
\(\tau_k\to\tau\).
\end{proposition}
\begin{proof}
Fix \(m\), suppress its subscript, and use the local coefficients in
\cref{eq:local-frozen-coefficients} and the transition matrices
\(\Lambda_m,\Phi_m^{\mathrm{cl}}\) in
\cref{eq:explicit-response-gains}. On the fixed square \([t_0,T]^2\), the
matrix kernel of \(\widehat J^\tau\) is
\begin{equation}\label{eq:extended-response-kernel}
\begin{aligned}
j_\tau(t,r)=&-\mathbf1_{\{t,r\ge\tau\}}R_m^{-1}B_m^\top
\bigg[\Pi^{\rm fr}(t)\bigg(
\mathbf1_{\{r\le t\}}\Phi^{\rm cl}(t,r)\mathsf E^{\rm fr}-\int_\tau^{\min\{t,r\}}\Phi^{\rm cl}(t,s)\mathsf D^{\rm fr}\\
&
\Lambda(s,r)\Pi^{\rm fr}(r)\mathsf E^{\rm fr}\,ds\bigg)+\mathbf1_{\{t\le r\}}\Lambda(t,r)\Pi^{\rm fr}(r)\mathsf E^{\rm fr}
\bigg].
\end{aligned}
\end{equation}
Define this kernel to be zero unless \(t,r\ge\tau\).
The formula follows from \cref{eq:explicit-response-gains} by Fubini.
The local chart and transition matrices are bounded on compact time
domains, giving one uniform Frobenius bound for \(j_\tau\).
If \(\tau_k\to\tau\), then \(j_{\tau_k}\to j_\tau\) almost
everywhere on the fixed square: only its boundary lines are excluded,
and the integral is continuous in its lower endpoint. Dominated
convergence therefore gives
\[
\|\widehat J^{\tau_k}-\widehat J^\tau\|_{\rm HS}^2
=\int_{t_0}^T\int_{t_0}^T
\|j_{\tau_k}(t,r)-j_\tau(t,r)\|_{\rm F}^2\,dr\,dt\longrightarrow0.
\]
The zero-extended state gain \(\widehat K_m^\tau\) is multiplication of a finite-dimensional
vector by the matrix-valued function
\[
k_\tau(t)=-\mathbf1_{\{t\ge\tau\}}R_m^{-1}B_m^\top
\Pi^{\rm fr}(t)\Phi^{\rm cl}(t,\tau)\mathfrak P_m.
\]
The same argument gives \(L^2\) Frobenius convergence of \(k_\tau\),
hence operator-norm convergence of \(\widehat K_m^\tau\).
For \(\widehat h_m^\tau\), \(g_m^0\) is independent of \(\tau\)
and \(X_m^0=-\int_\tau^t\Phi_m^{\rm cl}(t,s)\mathsf D_m^{\rm fr}g_m^0(s)\,ds\)
is continuous in the lower endpoint. Uniform bounds again give
\(L^2\) convergence by domination, including \(\tau=T\), where
all three zero-extended maps vanish.
\end{proof}

\section{Calculations for the divergent-plan counterexample}
\label{app:divergence-details}

For the model of \cref{prop:zeno-strategy-counterexample}, put
\(H=2-\tau\). The global and actor-first local state blocks are
\[
U(\tau)=\begin{pmatrix}1+H&-4H\\4H&1+H\end{pmatrix},\quad
U_p(\tau)=\begin{pmatrix}1+H&-4H\\0&1\end{pmatrix},\quad
U_q(\tau)=\begin{pmatrix}1+H&4H\\0&1\end{pmatrix}.
\]
Their determinants are \((1+H)^2+16H^2\) and \(1+H\), respectively,
so both boundary conditions hold. With
\(d_p=-(z_p-4z_q)/(1+H)\), \(d_q=-(z_q+4z_p)/(1+H)\), and
\(\lambda_p=-\lambda_q=4H/(1+H)\), the constant-plan fixed point solves
\[
a=d_p+\lambda_pb,\qquad
\bigl(1+(4H/(1+H))^2\bigr)b=d_q+\lambda_qd_p.
\]
It is unique; the cycle gain tends to \(-4\) as \(\tau\to1\).

The initial plans also satisfy the belief convention
\cref{eq:plan-generating-belief}. At time zero, the equilibrium
generator for a belief \(b\) is the constant pair
\[
\begin{pmatrix}u_1[b]\\u_2[b]\end{pmatrix}
=-\frac1{73}\begin{pmatrix}35&-4\\4&35\end{pmatrix}b.
\]
For \(X_0=(0,-s_1b_0)^\top\), take
\[
b_0^1=(0,0)^\top,\qquad
b_0^2=\bigl((35s_1-73)b_0/4,-s_1b_0\bigr)^\top.
\]
Each belief preserves its population's true initial mean. The selected
plans are \(u_1^0=0\) and \(u_2^0=b_0\), so execution up to \(s_1\)
gives \(X_{s_1}=0\). Thus the nonzero first response is genuine.

It remains to justify the induction in \cref{eq:zeno-counterexample-cone}.
Use the schedule and amplitudes of \cref{eq:zeno-counterexample-amplitudes},
and put \(S_k=|z_p(s_k)|+|z_q(s_k)|\).
The base case is \(S_1=0\). For all \(k\),
\(2\le\lambda_{p,k}\le2.1,\quad 2\le|\lambda_{q,k}|\le2.1\), and both execution
intervals have length at most \(1/128\).
If \(S_k\le0.1|b_{k-1}|\), then
\[
|d_{p,k}|\le2S_k\le0.2|b_{k-1}|,\qquad
1.8|b_{k-1}|\le|a_k|\le2.3|b_{k-1}|.
\]
Propagation to \(t_k\) gives
\[
|z_p(t_k)|+|z_q(t_k)|
\le\left(0.1+\frac{2.3+1}{128}\right)|b_{k-1}|
<0.13|b_{k-1}|,
\]
so \(|d_{q,k}|<0.26|b_{k-1}|\). Consequently,
\[
3|b_{k-1}|<(2\cdot1.8-0.26)|b_{k-1}|
<|b_k|<(2.1\cdot2.3+0.26)|b_{k-1}|<5.1|b_{k-1}|.
\]
Propagation on the second interval closes the induction:
\[
S_{k+1}\le
\left(0.1+\frac{2.3+1}{128}+\frac{2.3+5.1}{128}\right)|b_{k-1}|
<0.3|b_{k-1}|<0.1|b_k|.
\]
The strict dominance of \(\lambda_{p,k}b_{k-1}\) over \(d_{p,k}\),
and of \(\lambda_{q,k}a_k\) over \(d_{q,k}\), gives
\(\operatorname{sgn}(a_k)=\operatorname{sgn}(b_{k-1})\) and
\(\operatorname{sgn}(b_k)=-\operatorname{sgn}(b_{k-1})\).
Thus every scheduled update is genuine (the first by construction),
with \(|b_k|>3^k|b_0|\) and \(|a_k|\ge1.8\,3^{k-1}|b_0|\).
The upper bounds above give the executed-control estimate already
displayed in the main proof, since \(5.1^2/64<1\).

\label{end:appendices}
\begingroup
\raggedbottom
\interlinepenalty=10000
\bibliographystyle{siamplain}
\bibliography{references}

@inproceedings{aggarwal-zaman-basar-2022,
 author={S. Aggarwal and M. A. {uz Zaman} and T. Ba\c{s}ar},
 title={Linear quadratic mean-field games with communication constraints},
 booktitle={Proc. 2022 American Control Conference (ACC)}, publisher={IEEE},
 year={2022}, pages={1323--1329}, doi={10.23919/ACC53348.2022.9867728}}

@book{basar-olsder-1999,
 author={T. Basar and G. J. Olsder}, title={Dynamic Noncooperative Game Theory},
 edition={2nd}, series={Classics in Applied Mathematics}, publisher={SIAM},
 address={Philadelphia}, year={1999}}

@article{bensoussan-feng-huang-2021,
 author={A. Bensoussan and X. Feng and J. Huang},
 title={Linear-quadratic-{Gaussian} mean-field-game with partial observation and common noise},
 journal={Math. Control Relat. Fields}, volume={11}, year={2021}, pages={23--46}, doi={10.3934/mcrf.2020025}}

@article{bensoussan-huang-lauriere-2018,
 author={A. Bensoussan and T. Huang and M. Lauri\`ere},
 title={Mean field control and mean field game models with several populations},
 journal={Minimax Theory Appl.}, volume={3}, year={2018}, pages={173--209}}

@misc{bertucci-2022,
 author={C. Bertucci}, title={Mean field games with incomplete information},
 howpublished={arXiv preprint arXiv:2205.07703}, year={2022}}

@book{bertsekas-tsitsiklis-1989,
 author={D. P. Bertsekas and J. N. Tsitsiklis},
 title={Parallel and Distributed Computation: Numerical Methods},
 publisher={Prentice-Hall}, address={Englewood Cliffs, NJ}, year={1989}}

@article{cardaliaguet-hadikhanloo-2017,
 author={P. Cardaliaguet and S. Hadikhanloo},
 title={Learning in mean field games: The fictitious play},
 journal={ESAIM Control Optim. Calc. Var.}, volume={23}, year={2017}, pages={569--591}, doi={10.1051/cocv/2016004}}

@book{carmona-delarue-2018,
 author={R. Carmona and F. Delarue},
 title={Probabilistic Theory of Mean Field Games with Applications I: Mean Field FBSDEs, Control, and Games},
 series={Probability Theory and Stochastic Modelling}, publisher={Springer}, address={Cham}, year={2018}}

@article{casgrain-jaimungal-2020,
 author={P. Casgrain and S. Jaimungal},
 title={Mean-field games with differing beliefs for algorithmic trading},
 journal={Math. Finance}, volume={30}, year={2020}, pages={995--1034}, doi={10.1111/mafi.12237}}

@article{darouach-2000,
 author={M. Darouach}, title={Existence and design of functional observers for linear systems},
 journal={IEEE Trans. Automat. Control}, volume={45}, year={2000}, pages={940--943}}

@article{degond-herty-liu-2017,
 author={P. Degond and M. Herty and J.-G. Liu}, title={Meanfield games and model predictive control},
 journal={Commun. Math. Sci.}, volume={15}, year={2017}, pages={1403--1422}, doi={10.4310/CMS.2017.v15.n5.a9}}

@article{douglas-1966,
 author={R. G. Douglas}, title={On majorization, factorization, and range inclusion of operators on {Hilbert} space},
 journal={Proc. Amer. Math. Soc.}, volume={17}, year={1966}, pages={413--415}}

@article{feng-huang-jia-2025,
 author={X. Feng and J. Huang and Y. Jia}, title={Robust linear--quadratic mean-field-team with heterogeneous beliefs},
 journal={Systems Control Lett.}, volume={206}, year={2025}, pages={106266}, doi={10.1016/j.sysconle.2025.106266}}

@article{fernando-trinh-jennings-2010,
 author={T. L. Fernando and H. M. Trinh and L. Jennings},
 title={Functional observability and the design of minimum order linear functional observers},
 journal={IEEE Trans. Automat. Control}, volume={55}, year={2010}, pages={1268--1273}}

@article{firoozi-pakniyat-caines-2022,
 author={D. Firoozi and A. Pakniyat and P. E. Caines},
 title={A class of hybrid {LQG} mean field games with state-invariant switching and stopping strategies},
 journal={Automatica}, volume={141}, year={2022}, pages={110244}, doi={10.1016/j.automatica.2022.110244}}

@book{fudenberg-levine-1998,
 author={D. Fudenberg and D. K. Levine}, title={The Theory of Learning in Games},
 publisher={MIT Press}, address={Cambridge, MA}, year={1998}}

@book{goebel-sanfelice-teel-2012,
 author={R. Goebel and R. G. Sanfelice and A. R. Teel},
 title={Hybrid Dynamical Systems: Modeling, Stability, and Robustness},
 publisher={Princeton University Press}, address={Princeton, NJ}, year={2012}}

@article{huang-caines-malhame-2007,
 author={M. Huang and P. E. Caines and R. P. Malhame},
 title={Large-population cost-coupled {LQG} problems with nonuniform agents: Individual-mass behavior and decentralized {$\epsilon$}-{Nash} equilibria},
 journal={IEEE Trans. Automat. Control}, volume={52}, year={2007}, pages={1560--1571}}

@misc{jin-ren-yao-zhang-2026,
 author={Y. Jin and L. Ren and W. Yao and X. Zhang},
 title={Linear--quadratic mean field games under heterogeneous erroneous initial information},
 howpublished={arXiv preprint arXiv:2409.09375v4}, year={2026}, doi={10.48550/arXiv.2409.09375}}

@book{kailath-1980,
 author={T. Kailath}, title={Linear Systems}, publisher={Prentice-Hall},
 address={Englewood Cliffs, NJ}, year={1980}}

@article{lasry-lions-2007,
 author={J.-M. Lasry and P.-L. Lions}, title={Mean field games},
 journal={Japanese J. Math.}, volume={2}, year={2007}, pages={229--260}}

@article{li-nie-wang-yan-2024,
 author={M. Li and T. Nie and S. Wang and K. Yan},
 title={Incomplete information mean-field games and related {Riccati} equations},
 journal={J. Optim. Theory Appl.}, volume={203}, year={2024}, pages={2487--2508}, doi={10.1007/s10957-024-02508-0}}

@book{liberzon-2003,
 author={D. Liberzon}, title={Switching in Systems and Control},
 series={Systems \& Control: Foundations \& Applications},
 publisher={Birkh\"auser}, address={Boston}, year={2003}}

@book{ljung-1999,
 author={L. Ljung}, title={System Identification: Theory for the User},
 edition={2nd}, publisher={Prentice Hall PTR}, address={Upper Saddle River, NJ}, year={1999}}

@article{moll-ryzhik-2026,
 author={B. Moll and L. Ryzhik}, title={Mean field games without rational expectations},
 journal={Commun. Contemp. Math.}, volume={28}, number={5}, year={2026}, pages={2640007}, doi={10.1142/S0219199726400079}}

@article{montanari-duan-aguirre-motter-2022,
 author={A. N. Montanari and C. Duan and L. A. Aguirre and A. E. Motter},
 title={Functional observability and target state estimation in large-scale networks},
 journal={Proc. Natl. Acad. Sci. USA}, volume={119}, year={2022}, pages={e2113750119}}

@article{mouzouni-2020,
 author={C. Mouzouni}, title={On quasi-stationary mean field games models},
 journal={Appl. Math. Optim.}, volume={81}, year={2020}, pages={655--684}, doi={10.1007/s00245-018-9484-y}}

@article{neumann-2024,
 author={B. A. Neumann}, title={A myopic adjustment process for mean field games with finite state and action space},
 journal={Internat. J. Game Theory}, volume={53}, year={2024}, pages={159--195}, doi={10.1007/s00182-023-00866-z}}

@incollection{nisan-schapira-zohar-2008,
 author={N. Nisan and M. Schapira and A. Zohar}, title={Asynchronous best-reply dynamics},
 booktitle={Internet and Network Economics}, series={Lecture Notes in Computer Science},
 volume={5385}, publisher={Springer}, address={Berlin}, year={2008}, pages={531--538}}

@article{rodrigues-oliveira-krstic-basar-2026,
 author={V. H. P. Rodrigues and T. R. Oliveira and M. Krsti\'c and T. Ba\c{s}ar},
 title={Distributed event-triggered {Nash} equilibrium seeking},
 journal={Automatica}, volume={192}, year={2026}, pages={113164}, doi={10.1016/j.automatica.2026.113164}}

@article{sen-caines-2019,
 author={{\c S}en, N. and Caines, P. E.}, title={Mean field games with partial observation},
 journal={SIAM J. Control Optim.}, volume={57}, year={2019}, pages={2064--2091}, doi={10.1137/17M1140133}}

@misc{yang-2026,
 author={S. Yang}, title={Mean-field reinforcement learning without synchrony},
 howpublished={arXiv:2602.18026v1}, year={2026}, url={https://arxiv.org/abs/2602.18026}}
\endgroup

\end{document}